\documentclass[11pt]{article}

\usepackage{amsthm,amssymb,amsfonts,amsmath}
\usepackage{amscd} 
\usepackage{mathrsfs}
\usepackage{biblatex}
\usepackage{graphicx}

\usepackage{color}

\providecommand{\Z}{}

\providecommand{\C}{}

\renewcommand{\Z}{\mathbb{Z}}

\renewcommand{\C}{\mathbb{C}}

\newcommand{\ad}{\mathrm{ad}}

\providecommand{\ora}[1]{}
\renewcommand{\ora}[1]{\overrightarrow{#1}}

\newtheorem{thm}{Theorem}
\newtheorem{lem}[thm]{Lemma}
\newtheorem{prop}[thm]{Proposition}

\newtheorem{dfn}[thm]{Definition}

\newtheorem{fact}[thm]{Fact}

\numberwithin{thm}{section}
\numberwithin{equation}{section}

\newtheorem{prob}{Problem}
\usepackage{graphicx}
\usepackage{algpseudocode}
\usepackage{amssymb}
\usepackage{hyperref}
\usepackage{tikz}

\usepackage{tikz-cd}
\usetikzlibrary{shapes.geometric, automata, arrows.meta}

\begin{document}
\author{Jacob Canel\\ Department of Mathematics, Pennsylvania State University, \\ \href{mailto:jmc8684@psu.edu}{jmc8684@psu.edu}}
\date{\today}

\title{Diagrams for Roots and Quiver Representations in the context of the Gabriel-Kac Theorem}
\maketitle
\begin{abstract}
    Using the data of the Dynkin diagrams, we build graphical calculi for representing the positive roots of Kac-Moody algebras and the representations of quivers. Once developed, we use these to analyze the asymptotic size of the root and representation spaces, and discuss how this relates to the Gabriel-Kac theorem, the correspondence between these two structures. We give a refinement of this correspondence in terms of the trichotomy of growth rates for Kac-Moody algebras depending on if they are of finite, affine, or indefinite type, which is done in terms of what we call ``tree like" representations.
\end{abstract}
\tableofcontents
\section{Introduction}
There is an intimate relationship between the Kac-Moody algebra defined on a simply-laced graph and the representation theory of any quiver produced by picking an orientation on that graph. In particular, Gabriel \cite{Gabriel_1972} showed that roots of an $ADE$ Lie algebra correspond $1-1$ with representations of such a quiver, and Kac \cite{KacRoots} showed that existence of a root corresponded to existence of an indecomposible quiver representation in the case of arbitrary Kac-Moody algebras. Nakajima developed this line of reasoning further by using homological constructions on his quiver varieties to produce \textit{representations} of the Kac-Moody algebra. In the same paper where he demonstrated his part of the correspondence between roots and representations (which we will call the Gabriel-Kac theorem), Kac conjectured that the constant part of the polynomial count of absolutely indecomposable quiver representations was the multiplicity of the associated root, which was eventually proven by Hausel \cite{Hausel2010} using Nakajima's quiver varieties.

We add to this discussion by presenting a graphical calculi suitable to representing both roots and quiver representations, which differ minimally, and use them to illustrate this relationship more intuitively. We introduce a language for generating Kac-Moody roots and use it to give a simple proof of the growth-rate trichotomy for Kac-Moody algebras \cite{Kac_2003}. This language is defined entirely in terms of the Dynkin diagrams, and reflect their combinatorial structure. We further introduce a similar language for generating quiver representations (over $\C$), and discuss the families of representations thus generated which naturally correspond to the roots, as in Kac's conjectures. These languages are given solely from the data of the Dynkin diagrams, which is the conceptual contribution of this work.

In \S\ref{sec:Background}, we build the visual calculi for both sets of objects in the correspondence. In \S\ref{sec:Simple}, \S\ref{sec:Affine}, and \S\ref{sec:WildBinary}, we discuss the diagrammatics and asymptotic behavior of simple, affine, and wild-type Kac-Moody algebras respectively. In \S\ref{sec:finitequiver}, \S\ref{sec:affinequiver}, and \S\ref{sec:WildBinary}, we discuss the diagrams and asymptotic behavior of quiver representations for the finite, affine, and wild-type quivers respectively. Finally, we make some concluding remarks in \S\ref{sec:conc}.

\section{Defining the Graphical Calculi}\label{sec:Background}

\subsection{Building the Graphical Calculus for Kac-Moody Roots}

For background on constructing Lie algebras from Dynkin diagrams, please consult the appendix \ref{sec:farbackground}. Every element of $E$ (the positive part of the Lie algebra) can be written as an $\mathrm{ad}$-polynomial, meaning a sum of terms like $[e_{i_n},[e_{i_{n-1}},[e_{i_{n-2}},[...,e_{i_1}]...]]]]$. We can simplify this by just writing their subscripts in sequence (in reverse order $i_1,i_2... i_n$) or by representing them using Dynkin diagram vertices, which is what we will do. If two generators have commuting labels, we can just as well write them one on top of the other, since we can move one past the other freely. If two generators in adjacent positions don't commute, there is at least one edge between them in the Dynkin graph. An edge should be drawn between them in our diagram to indicate this.

We will therefore write elements of the Lie algebra as graphs sorted into levels. Since $\ad$-polynomials have some initial vector upon which repeated $\ad$-actions are taken, there is only one vertex in the initial level (which we write to the left). We can also move every vertex in the graph towards the left until it reaches something with which it does not commute, so we can assert that every vertex in every other level must have a line to something in the previous level. 

We will call the $a$-valence of a set of vertices the sum of the number of directed edges that go from those vertices (counted with multiplicity) to the vertex $a$ in the Dynkin diagram and counting edges labeled $(x,y)$ with the $x$ towards the $a$ as $a$ directed edges moving towards $a$. If the number of $a$-vertices in level $n+1$ exceeds the $a$-valence of level $n$, it is possible to move at least one $a$ vertex to level $n-1$, so we will insist on disallowing such arrangements. These rules do not provide a complete description of a basis for $E$, as not all relations are applied when they could be, but it does reduce the search space greatly, especially in the semisimple and near semisimple cases.

The sections \ref{sec:Simple}, \ref{sec:Affine}, and \ref{sec:WildBinary} are devoted to following a method which proceeds thusly:
\begin{enumerate}
  
    \item Iteratively create new candidate elements of $E$ using the results above
    \item Apply the relations to reduce the number of candidate elements
    \item Describe the resulting long-time behavior
\end{enumerate}

The resulting diagrams, when viewed as a set of instructions for how to build the corresponding element of the Kac-Moody algebra, provide half of the story about the roots. Namely, they provide an exhaustive list of candidates for different root vectors. This only gives an upper bound for the root space dimension, and evaluating the list of candidates using:
\begin{enumerate}
    \item the Serre relations which assert that specific generators commute, which are entirely manifest in this description,
    \item the other Serre relations, which are partially manifest in this description,
    \item the Jacobi identity, which is not manifest in this description,
    \item and anticommutativity of the Lie bracket, which here shows up as the fact $ad_{x_1}ad_{x_2}...ad_{b}(a) = -ad_{x_1}ad_{x_2}...ad_{a}(b)$, which leads to a simple combinatorial move in this description.
\end{enumerate}
That a particular root space has non-zero multiplicity can be validated using the Gabriel-Kac theorem. Indeed, this graphical calculus lends itself naturally to generating the quiver representations.
\subsection{The Calculus for Quivers}
Quivers are perhaps the most egregious example in mathematics of something that is very easy to define, but difficult to study and understand.
\begin{dfn}
    A \textbf{quiver} is a directed multigraph, i.e. it is a pair of sets $V$ and $E$ with source and target maps $s,t: E\rightarrow V$
\end{dfn}
What gives life to quivers is the following:
\begin{dfn}
    A \textbf{representation} of the quiver $(V,E, s,t)$ is a collection of vector spaces $W_v$ for each $v\in V$ and linear maps $W_e: W_{s(e)}\rightarrow W_{t(e)}$ for each edge. Two representations $U$ and $W$ are \textbf{isomorphic} if there are linear isomorphisms $F_v: W_{v}\rightarrow U_v$ such that $F_{t(e)} \circ W_e = U_e\circ F_{s(e)}$ for each $v\in V$ and $e\in E$.
\end{dfn}
While these are both the correct definitions, there are easier ways to see what is going on. The simplest case of Kac-Moody algebras is the $A_n$ case, meaning those given by Dynkin diagrams which are paths of length $n$. The Lie algebras represented by these are the complex special linear algebras $sl_{n+1}$. Correspondingly, the quivers with underlying graph $A_n$ have the simplest class of indecomposable representations. We may take the $A_n$ diagram as an example
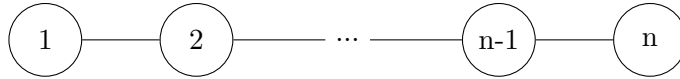
\begin{figure}[h!]\label{fig:A_n}
    \centering
    \begin{tikzpicture}[auto, node distance=2cm, transform shape]
            \node[state] (1) {1};
                \node[state] (2) [right of = 1] {2};
                \node (3) [right of = 2] {...};
                \node[state] (4) [right of = 3] {n-1};
                \node[state] (5) [right of = 4] {n};

                \draw (1) to node[left] {} (2);
                \draw (2) to node[left] {} (3);
                \draw (3) to node[left] {} (4);
                \draw (4) to node[left] {} (5);
        \end{tikzpicture}
    \caption{The $A_n$ Dynkin Diagram}
    
\end{figure}\\
We might, for example, give it the following orientation:\\
\begin{figure}[h]\label{fig:A_n quiver}
    \centering
    \begin{tikzpicture}[auto, node distance=2cm, transform shape]
            \node[state] (1) {1};
                \node[state] (2) [right of = 1] {2};
                \node (3) [right of = 2] {...};
                \node[state] (4) [right of = 3] {n-1};
                \node[state] (5) [right of = 4] {n};

                \draw [arrows = {-Stealth[scale=2]}]  (1)  -- (2);
                \draw [arrows = {-Stealth[scale=2]}]  (2)  -- (3);
                \draw [arrows = {-Stealth[scale=2]}]  (3)  -- (4);
                 \draw [arrows = {-Stealth[scale=2]}]  (4)  -- (5);
        \end{tikzpicture}
    \caption{An $A_n$ Quiver}
\end{figure}
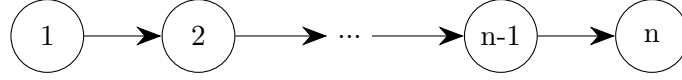\\
Constructing a representation of that quiver amounts to picking dimensions for the vector spaces corresponding to each node, and then picking matrices. We will represent this by drawing duplicates of the same node and labelling the edges with matrix coefficients. We would see for example that if we tried to build a representation of the $A_6$ quiver, so oriented, with dimension vector $(0,1,2,1,1,1)$, we would get a system which looks like this:\\
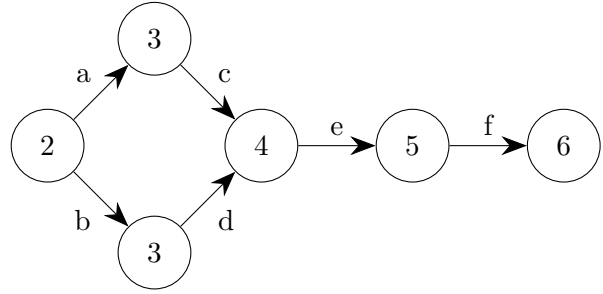
\begin{figure}[h]\label{fig:A_n quiver}
    \centering
    \begin{tikzpicture}[auto, node distance=2cm, transform shape]
            \node[state] (2) {2};
                \node[state] (3) [above right of = 1] {3};
                \node[state] (3too) [below right of = 1] {3};
                \node[state] (4) [below right of = 3] {4};
                \node[state] (5) [right of = 4] {5};
                \node[state] (6) [right of = 5] {6};

                \draw [arrows = {-Stealth[scale=2]}]  (2) -- (3) node[midway, above left]{a};
                \draw [arrows = {-Stealth[scale=2]}]  (2)  -- (3too) node[midway, below left]{b};
                \draw [arrows = {-Stealth[scale=2]}]  (3)  -- (4) node[midway, above right]{c};
                 \draw [arrows = {-Stealth[scale=2]}]  (3too)  -- (4) node[midway, below right]{d};
                 \draw [arrows = {-Stealth[scale=2]}]  (4)  -- (5) node[midway, above]{e};;
                 \draw [arrows = {-Stealth[scale=2]}]  (5)  -- (6) node[midway, above]{f};;
                 
        \end{tikzpicture}
    \caption{The Space of $(0,1,2,1,1,1)$ $A_6$ quiver representations}
\end{figure}\\
Really, we only consider representations up to a choice of basis (equivalently \textit{isomorphism}). Generically, the image of the vector space $2$ is one dimensional, as is the kernel of the map $3\rightarrow 4$. So we may choose a basis of $3$ such that one element is in the kernel and the other is the image of the basis element chosen for $2$. This yields a system of the following kind:\\
\begin{figure}[h!]\label{fig:A_n quiver}
    \centering
    \begin{tikzpicture}[auto, node distance=2cm, transform shape]
            \node[state] (2) {2};
                \node[state] (3) [above right of = 1] {3};
                \node[state] (3too) [below right of = 1] {3};
                \node[state] (4) [below right of = 3] {4};
                \node[state] (5) [right of = 4] {5};
                \node[state] (6) [right of = 5] {6};

                \draw [arrows = {-Stealth[scale=2]}]  (2) -- (3) node[midway, above left]{1};
                \draw [arrows = {-Stealth[scale=2]}]  (2)  -- (3too) node[midway, below left]{0};
                \draw [arrows = {-Stealth[scale=2]}]  (3)  -- (4) node[midway, above right]{c};
                 \draw [arrows = {-Stealth[scale=2]}]  (3too)  -- (4) node[midway, below right]{0};
                 \draw [arrows = {-Stealth[scale=2]}]  (4)  -- (5) node[midway, above]{e};;
                 \draw [arrows = {-Stealth[scale=2]}]  (5)  -- (6) node[midway, above]{f};;
                 
        \end{tikzpicture}
    \caption{The Space of $(0,1,2,1,1,1)$ $A_6$ quiver subject to generic rank constraints}
\end{figure}
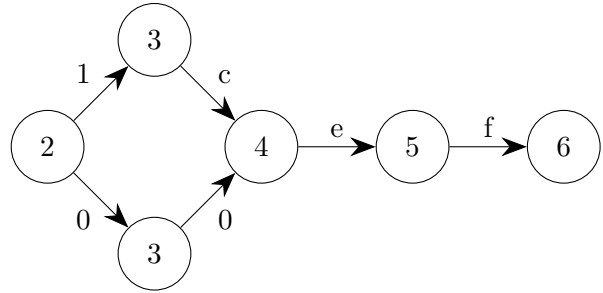\\[1cm]
We will, from now on, omit zero edges. This allows us to say the following:
\begin{prop}
    A quiver representation is decomposable if there is a choice of basis which renders the associated labeled diagram (with zero edges omitted) disconnected.
\end{prop}
\begin{proof}
    A quiver representation is decomposable iff is isomorphic to a nontrivial direct sum. For a choice of basis subordinate to this direct sum, any basis vector in one part of the direct sum cannot be connected to any basis vector in another part. Alternately, if the diagram is disconnected, the representation may be written as the direct sum of the connected components.
\end{proof}

Our diagram now looks like:\\
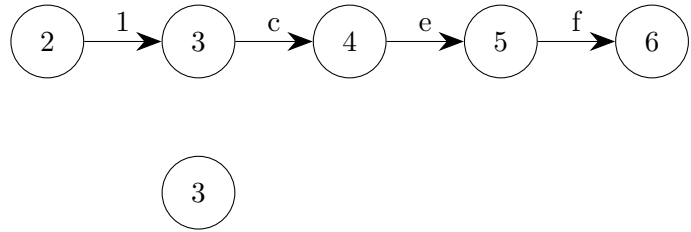
\begin{figure}[h!]\label{fig:A_n quiver}
    \centering
    \begin{tikzpicture}[auto, node distance=2cm, transform shape]
            \node[state] (2) {2};
                \node[state] (3) [right of = 2] {3};
                \node[state] (3too) [below of = 3] {3};
                \node[state] (4) [right of = 3] {4};
                \node[state] (5) [right of = 4] {5};
                \node[state] (6) [right of = 5] {6};

                \draw [arrows = {-Stealth[scale=2]}]  (2) -- (3) node[midway, above ]{1};
              
                \draw [arrows = {-Stealth[scale=2]}]  (3)  -- (4) node[midway, above ]{c};
                
                 \draw [arrows = {-Stealth[scale=2]}]  (4)  -- (5) node[midway, above]{e};;
                 \draw [arrows = {-Stealth[scale=2]}]  (5)  -- (6) node[midway, above]{f};;
                 
        \end{tikzpicture}
    \caption{The Space of $(0,1,2,1,1,1)$ $A_6$ quiver subject to generic rank constraints}
\end{figure}\\
We may now make scaling choices for the basis elements of $4$, $5$, and $6$, and the system looks generically like:\\
\begin{figure}[h!]\label{fig:A_n quiver}
    \centering
    \begin{tikzpicture}[auto, node distance=2cm, transform shape]
            \node[state] (2) {2};
                \node[state] (3) [right of = 2] {3};
                \node[state] (3too) [below of = 3] {3};
                \node[state] (4) [right of = 3] {4};
                \node[state] (5) [right of = 4] {5};
                \node[state] (6) [right of = 5] {6};

                \draw [arrows = {-Stealth[scale=2]}]  (2) -- (3) node[midway, above ]{1};
              
                \draw [arrows = {-Stealth[scale=2]}]  (3)  -- (4) node[midway, above ]{1};
                
                 \draw [arrows = {-Stealth[scale=2]}]  (4)  -- (5) node[midway, above]{1};;
                 \draw [arrows = {-Stealth[scale=2]}]  (5)  -- (6) node[midway, above]{1};;
                 
        \end{tikzpicture}
    \caption{The generic representation of dimension $(0,1,2,1,1,1)$ in the leftward oriented $A_6$ quiver}
\end{figure}
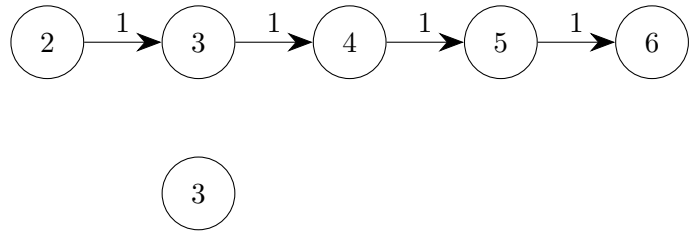\\
The fact that this quiver is disconnected reflects the fact that there are no indecomposable representations of (any) $A_n$ quiver of these dimensions. The fact that the representation generically decomposes in this way is orientation dependent, as in the case where we two arrows from one dimensional spaces into a two dimensional space, these are generically not linearly dependent, and thus the graph would not have a single isolated vertex after being simplified as much as possible.
The following fact is important to the combinatorics of the quiver representations:
\begin{lem}\label{lem:treelike}
    If the graph of a quiver representation diagram is a forest, all of the nonzero coefficients may be set to $1$ by scaling the chosen basis.
\end{lem}
\begin{dfn}
    We call an indecomposable quiver representation \textbf{tree-like} if there is a choice of basis which makes its underlying quiver representation diagram graph a tree.
\end{dfn}

We will find that tree-like quiver representations are important, as they correspond to the roots of Kac-Moody algebras.

\section{Simple Lie Algebras: Diagrams for Highest Roots}\label{sec:Simple}
\subsection{The Infinite Families $ABCD$}
The simplest case are the classical infinite families $A_n$, $B_n$, $C_n$, and $D_n$. In each of these cases, the process of searching through candidate roots terminates, and we eventually reach an inextensible end state. The process is, as expected, easiest for these very classical cases. In the $A_n$ case, all roots are sub-paths of the Dynkin diagram itself, and all other cases are simple modifications of this, with the major caveat being that the presence of a special feature, such as a double edge or triple point, allows one to "bounce" back while constructing these intervals. In the $C_n$ case this allows for a full doubling of the Dynkin diagram to be used, whereas in the $B_n$ and $D_n$ cases (those corresponding to the special orthogonal groups) the doubling terminates one root short of completion. The results may be seen in figure \ref{fig:simpleinfinite}.
\begin{figure}[h!]
    \centering

    \begin{tabular}{|c|c|c|}
    \hline
        Algebra Name & Dynkin Diagram & Highest Root Pattern \\ \hline
        $A_n$  & \begin{tikzpicture}[auto, node distance=2cm, scale = 0.4, transform shape]
            \node[state] (1) {1};
                \node[state] (2) [right of = 1] {2};
                \node (3) [right of = 2] {...};
                \node[state] (4) [right of = 3] {n-1};
                \node[state] (5) [right of = 4] {n};

                \draw (1) to node[left] {} (2);
                \draw (2) to node[left] {} (3);
                \draw (3) to node[left] {} (4);
                \draw (4) to node[left] {} (5);
        \end{tikzpicture} & \begin{tikzpicture}[auto, node distance=2cm, scale = 0.4, transform shape]
            \node[state] (1) {1};
                \node[state] (2) [right of = 1] {2};
                \node (3) [right of = 2] {...};
                \node[state] (4) [right of = 3] {n-1};
                \node[state] (5) [right of = 4] {n};

                \draw (1) to node[left] {} (2);
                \draw (2) to node[left] {} (3);
                \draw (3) to node[left] {} (4);
                \draw (4) to node[left] {} (5);
        \end{tikzpicture}\\ \hline
        $B_n$  & \begin{tikzpicture}[auto, node distance=2cm, scale = 0.4, transform shape]
            \node[state] (1) {1};
                \node[state] (2) [right of = 1] {2};
                \node (3) [right of = 2] {...};
                \node[state] (4) [right of = 3] {n-1};
                \node[state] (5) [right of = 4] {n};

                \draw (1) to node[left] {} (2);
                \draw (2) to node[left] {} (3);
                \draw (3) to node[left] {} (4);
                \draw (4) to[bend left] node[] {} (5);
                \draw (4) to[bend right] node[above] {$>$} (5);
        \end{tikzpicture} & \begin{tikzpicture}[auto, node distance=1.5cm, scale = 0.25, transform shape]
            \node[state] (1) {1};
                \node[state] (2) [right of = 1] {2};
                \node (3) [right of = 2] {...};
                \node[state] (4) [right of = 3] {n-1};
                \node[state] (5) [below right of = 4] {n};
                \node[state] (6) [above right of = 4] {n};
                \node[state] (7) [below right of = 6] {n-1};
                \node[state] (8) [right of = 7] {n-2};
                \node (9) [right of = 8] {...};
                \node[state] (10) [right of = 9] {2};

                \draw (1) to node[left] {} (2);
                \draw (2) to node[left] {} (3);
                \draw (3) to node[left] {} (4);
                \draw (4) to node[left] {} (5);
                \draw (4) to node[left] {} (6);
                \draw (5) to node[left] {} (7);
                \draw (6) to node[left] {} (7);
                \draw (7) to node[left] {} (8);
                \draw (8) to node[left] {} (9);
                \draw (9) to node[left] {} (10);
        \end{tikzpicture}\\ \hline
        $C_n$  & \begin{tikzpicture}[auto, node distance=2cm, scale = 0.4, transform shape]
            \node[state] (1) {1};
                \node[state] (2) [right of = 1] {2};
                \node (3) [right of = 2] {...};
                \node[state] (4) [right of = 3] {n-1};
                \node[state] (5) [right of = 4] {n};

                \draw (1) to node[left] {} (2);
                \draw (2) to node[left] {} (3);
                \draw (3) to node[left] {} (4);
                \draw (4) to[bend left] node[] {} (5);
                \draw (4) to[bend right] node[above] {$<$} (5);
        \end{tikzpicture} & \begin{tikzpicture}[auto, node distance=2cm, scale = 0.25, transform shape]
            \node[state] (1) {1};
                \node[state] (2) [right of = 1] {2};
                \node (3) [right of = 2] {...};
                \node[state] (4) [right of = 3] {n-1};
                \node[state] (5) [right of = 4] {n};
                \node[state] (6) [right of = 5] {n-1};
                \node[state] (7) [right of = 6] {n-2};
                \node (8) [right of = 7] {...};
                \node[state] (9) [right of = 8] {1};

                \draw (1) to node[left] {} (2);
                \draw (2) to node[left] {} (3);
                \draw (3) to node[left] {} (4);
                \draw (4) to node[left] {} (5);
                \draw (5) to node[left] {} (6);
                \draw (6) to node[left] {} (7);
                \draw (7) to node[left] {} (8);
                \draw (8) to node[left] {} (9);
        \end{tikzpicture}
        \\ 
        \hline
        $D_n$  & \begin{tikzpicture}[auto, node distance=2cm, scale = 0.4, transform shape]
            \node[state] (1) {1};
                \node[state] (2) [right of = 1] {2};
                \node (3) [right of = 2] {...};
                \node[state] (4) [right of = 3] {n-2};
                \node[state] (5) [below right of = 4] {n-1};
                \node[state] (6) [above right of = 4] {n};

                \draw (1) to node[left] {} (2);
                \draw (2) to node[left] {} (3);
                \draw (3) to node[left] {} (4);
                \draw (4) to node[left] {} (5);
                \draw (4) to node[left] {} (6);
        \end{tikzpicture} & \begin{tikzpicture}[auto, node distance=1.5cm, scale = 0.25, transform shape]
            \node[state] (1) {1};
                \node[state] (2) [right of = 1] {2};
                \node (3) [right of = 2] {...};
                \node[state] (4) [right of = 3] {n-2};
                \node[state] (5) [below right of = 4] {n-1};
                \node[state] (6) [above right of = 4] {n};
                \node[state] (7) [below right of = 6] {n-2};
                \node[state] (8) [right of = 7] {n-3};
                \node (9) [right of = 8] {...};
                \node[state] (10) [right of = 9] {2};

                \draw (1) to node[left] {} (2);
                \draw (2) to node[left] {} (3);
                \draw (3) to node[left] {} (4);
                \draw (4) to node[left] {} (5);
                \draw (4) to node[left] {} (6);
                \draw (5) to node[left] {} (7);
                \draw (6) to node[left] {} (7);
                \draw (7) to node[left] {} (8);
                \draw (8) to node[left] {} (9);
                \draw (9) to node[left] {} (10);
        \end{tikzpicture}
        \\ 
        \hline

    \end{tabular}
    \caption{The highest root patterns of the infinite families.}
    \label{fig:simpleinfinite}
\end{figure}

There are two relationships here given by folding Dynkin diagrams using Kac's trick. The $B_n$ diagram is obtained from $D_{n+1}$ by performing the trick on the two final nodes, and the $C_n$ is obtained from $A_{2n-1}$ by applying the trick to the bilateral symmetry. The relationships between these algebras are shown clearly by the highest root patterns.

\subsection{The Finite Family $E$}
There is also a finite family of Dynkin graphs which represent simple algebras. These are the $E_6$, $E_7$, and $E_8$ graphs which round out the trinity of simply laced $ADE$ graphs
\begin{figure}[h]
    \centering
    \begin{tabular}{|c|c|c|}
    \hline
        Algebra Name & Dynkin Diagram & Highest Root Pattern \\ \hline
        $E_6$  & \begin{tikzpicture}[auto, node distance=2cm, scale = 0.3, transform shape]
            \node[state] (1) {1};
            \node[state] (2) [right of = 1] {2};
            \node[state] (3) [right of = 2] {3};
            \node[state] (x) [above of = 3] {x};
            \node[state] (4) [right of = 3] {4};
            \node[state] (5) [right of = 4] {5};
            \draw (1) to node[left] {} (2);
            \draw (2) to node[left] {} (3);
            \draw (3) to node[left] {} (4);
            \draw (4) to node[left] {} (5);
            \draw (3) to node[left] {} (x);

        \end{tikzpicture} & \begin{tikzpicture}[auto, node distance=2cm, scale = 0.3, transform shape]
                \node[state] (1) {x};
            
                \node[state] (2) [right of = 1] {3};
                
                \node[state] (3) [above right of = 2] {4};
                \node[state] (4) [below right of = 2] {2};
                
                \node[state] (5) [above right of = 3] {5};
                \node[state] (6) [below right of = 3] {3};
                \node[state] (7) [below right of = 4] {1};

                \node[state] (8) [above right of = 6] {4};
                \node[state] (9) [below right of = 6] {2};

                \node[state] (10) [above right of = 9] {3};

                \node[state] (11) [right of = 10] {x};

                \draw (1) to node[left] {} (2);

                \draw (2) to node[left] {} (3);
                \draw (2) to node[left] {} (4);

                \draw (3) to node[left] {} (5);
                \draw (3) to node[left] {} (6);

                \draw (4) to node[left] {} (6);
                \draw (4) to node[left] {} (7);

                \draw (8) to node[left] {} (5);
                \draw (8) to node[left] {} (6);
            
                \draw (9) to node[left] {} (6);
                \draw (9) to node[left] {} (7);

                \draw (8) to node[left] {} (10);
                \draw (9) to node[left] {} (10);

                \draw (10) to node[right] {} (11);
                
        \end{tikzpicture}\\ \hline
        $E_7$  & \begin{tikzpicture}[auto, node distance=2cm, scale = 0.3, transform shape]
            \node[state] (1) {1};
            \node[state] (2) [right of = 1] {2};
            \node[state] (3) [right of = 2] {3};
            \node[state] (4) [right of = 3] {4};
            \node[state] (5) [right of = 4] {5};
            \node[state] (x) [above of = 4] {x};
            \node[state] (6) [right of = 5] {6};
            \draw (1) to node[left] {} (2);
            \draw (2) to node[left] {} (3);
            \draw (3) to node[left] {} (4);
            \draw (4) to node[left] {} (5);
            \draw (4) to node[left] {} (x);
            \draw (5) to node[left] {} (6);

        \end{tikzpicture} & \begin{tikzpicture}[auto, node distance=1.5cm, scale = 0.3, transform shape]
            \node[state] (1) {x};
            
            \node[state] (2) [right of = 1] {4};
            
            \node[state] (3) [above right of = 2] {3};
            \node[state] (4) [below right of = 2] {5};

            \node[state] (5) [above right of = 3] {2};
            \node[state] (6) [above right of = 4] {4};
            \node[state] (7) [below right of = 4] {6};

            \node[state] (8) [above right of = 5] {1};
            \node[state] (9) [above right of = 6] {3};
            \node[state] (10) [above right of = 7] {5};
            
            \node[state] (11) [above right of = 9] {2};
            \node[state] (12) [above right of = 10] {4};
            
            \node[state] (13) [above right of = 12] {3};
            \node[state] (14) [below right of = 12] {x};

            \node[state] (15) [above right of = 14] {4};

            \node[state] (16) [right of = 15] {5};
            
            \node[state] (17) [right of = 16] {6};

            \draw (1) to node[left] {} (2);

            \draw (2) to node[left] {} (3);
            \draw (2) to node[left] {} (4);

            \draw (3) to node[left] {} (5);
            \draw (3) to node[left] {} (6);
            \draw (4) to node[left] {} (6);
            \draw (4) to node[left] {} (7);

            \draw (5) to node[left] {} (8);
            \draw (5) to node[left] {} (9);
            \draw (6) to node[left] {} (9);
            \draw (6) to node[left] {} (10);
            \draw (7) to node[left] {} (10);

            \draw (8) to node[left] {} (11);
            \draw (9) to node[left] {} (11);
            \draw (9) to node[left] {} (12);
            \draw (10) to node[left] {} (12);

            \draw (11) to node[left] {} (13);
            \draw (12) to node[left] {} (13);
            \draw (12) to node[left] {} (14);

            \draw (13) to node[left] {} (15);
            \draw (14) to node[left] {} (15);

            \draw (15) to node[left] {} (16);
            \draw (16) to node[left] {} (17);

        \end{tikzpicture}\\ \hline
        $E_8$  & \begin{tikzpicture}[auto, node distance=2cm, scale = 0.3, transform shape]
            \node[state] (1) {1};
            \node[state] (2) [right of = 1] {2};
            \node[state] (3) [right of = 2] {3};
            \node[state] (4) [right of = 3] {4};
            \node[state] (5) [right of = 4] {5};
            \node[state] (6) [right of = 5] {6};
            \node[state] (x) [above of = 5] {x};
            \node[state] (7) [right of = 6] {7};
            \draw (1) to node[left] {} (2);
            \draw (2) to node[left] {} (3);
            \draw (3) to node[left] {} (4);
            \draw (4) to node[left] {} (5);
            \draw (5) to node[left] {} (x);
            \draw (5) to node[left] {} (6);
            \draw (6) to node[left] {} (7);

        \end{tikzpicture} & \begin{tikzpicture}[auto, node distance=1.5cm, scale = 0.3, transform shape]
            \node[state] (1) {x};
            
            \node[state] (2) [right of = 1] {5};

            \node[state] (4) [above right of = 2] {6};
            \node[state] (5) [below right of = 2] {4};

            \node[state] (6) [above right of = 4] {7};
            \node[state] (7) [above right of = 5] {5};
            \node[state] (8) [below right of = 5] {3};

            \node[state] (9) [above right of = 7] {6};
            \node[state] (10) [above right of = 8] {4};
            \node[state] (11) [below right of = 8] {2};

            \node[state] (12) [above right of = 10] {5};
            \node[state] (13) [above right of = 11] {3};
            \node[state] (14) [below right of = 11] {1};

            \node[state] (15) [above right of = 12] {x};
            \node[state] (16) [above right of = 13] {4};
            \node[state] (17) [above right of = 14] {2};

            \node[state] (18) [above right of = 16] {5};
            \node[state] (19) [above right of = 17] {3};

            \node[state] (20) [above right of = 18] {6};
            \node[state] (21) [above right of = 19] {4};

            \node[state] (22) [above right of = 20] {7};
            \node[state] (23) [above right of = 21] {5};

            \node[state] (24) [above right of = 23] {6};
            \node[state] (25) [below right of = 23] {x};

            \node[state] (26) [above right of = 25] {5};

            \node[state] (27) [below right of = 26] {4};
            
            \node[state] (28) [above right of = 27] {3};
            
            \node[state] (29) [below right of = 28] {2};

            \node[state] (30) [above right of = 29] {1};
            
            \draw (1) to (2);

            \draw (2) to (4);
            \draw (2) to (5);

            \draw (4) to (6);
            \draw (4) to (7);
            \draw (5) to (7);
            \draw (5) to (8);

            \draw (6) to (9);
            \draw (7) to (9);
            \draw (7) to (10);
            \draw (8) to (10);
            \draw (8) to (11);

            \draw (9) to (12);
            \draw (10) to (12);
            \draw (10) to (13);
            \draw (11) to (13);
            \draw (11) to (14);

            \draw (12) to (15);
            \draw (12) to (16);
            \draw (13) to (16);
            \draw (13) to (17);
            \draw (14) to (17);

            \draw (15) to (18);
            \draw (16) to (18);
            \draw (16) to (19);
            \draw (17) to (19);

            \draw (18) to (20);
            \draw (18) to (21);
            \draw (19) to (21);

            \draw (20) to (22);
            \draw (20) to (23);
            \draw (21) to (23);

            \draw (22) to (24);
            \draw (23) to (24);
            \draw (23) to (25);

            \draw (24) to (26);
            \draw (25) to (26);

            \draw (26) to (27);

            \draw (27) to (28);

            \draw (28) to (29);

            \draw (29) to  (30);

        \end{tikzpicture}
        \\ 
        \hline
    \end{tabular}
    \caption{The $E$ family of Dynkin diagrams and their highest root patterns}
    \label{fig:EFamily}
\end{figure}
The $E_7$ and $E_8$ highest root patterns have tails which destroy the bilateral symmetry of the $E_6$ highest root diagram. The addition of an extra node, which would allow for bilaterally symmetric representatives, would also allow these patterns to be repeated ad infinitum, which we will see in full detail when discussing the affine algebras.

\subsection{$F_4$ and $G_2$}
The final two simple Lie algebras are given by Kac's trick on the $E_6$ and $D_4$ simply laced Lie algebras. They are not part of any family, finite or otherwise.
  \begin{figure}[h]
      \centering
   
    \begin{tabular}{|c|c|c|}
    \hline
        Algebra Name & Dynkin Diagram & Highest Root Pattern \\ \hline
        $F_4$  & \begin{tikzpicture}[auto, node distance=2cm, scale = 0.4, transform shape]
            \node[state] (1) {1};
                \node[state] (2) [right of = 1] {2};
                \node[state] (3) [right of = 2] {3};
                \node[state] (4) [right of = 3] {4};

                \draw (1) to node[left] {} (2);
                \draw (2) to[bend left] node[left] {} (3);
                \draw (2) to[bend right] node[above, yshift = 0.15cm] {$>$} (3);
                \draw (3) to node[left] {} (4);
        \end{tikzpicture} & \begin{tikzpicture}[auto, node distance=2cm, scale = 0.4, transform shape]
                \node[state] (1) {1};
            
                \node[state] (2) [right of = 1] {2};
                
                \node[state] (3) [above right of = 2] {3};
                \node[state] (4) [below right of = 2] {3};
                
                \node[state] (5) [above right of = 3] {4};
                \node[state] (6) [below right of = 3] {2};
                \node[state] (7) [below right of = 4] {4};

                \node[state] (8) [above right of = 6] {3};
                \node[state] (9) [below right of = 6] {3};

                \node[state] (10) [above right of = 9] {2};

                \node[state] (11) [right of = 10] {1};

                \draw (1) to node[left] {} (2);

                \draw (2) to node[left] {} (3);
                \draw (2) to node[left] {} (4);

                \draw (3) to node[left] {} (5);
                \draw (3) to node[left] {} (6);
                \draw (3) to node[left] {} (7);
                \draw (4) to node[left] {} (5);
                \draw (4) to node[left] {} (6);
                \draw (4) to node[left] {} (7);

                \draw (8) to node[left] {} (5);
                \draw (8) to node[left] {} (6);
                \draw (8) to node[left] {} (7);
                \draw (9) to node[left] {} (5);
                \draw (9) to node[left] {} (6);
                \draw (9) to node[left] {} (7);

                \draw (8) to node[left] {} (10);
                \draw (9) to node[left] {} (10);

                \draw (10) to node[right] {} (11);
                
        \end{tikzpicture}\\ \hline
        $G_2$  & \begin{tikzpicture}[auto, node distance=2cm, scale = 0.4, transform shape]
            \node[state] (1) {1};
                \node[state] (2) [right of = 1] {2};

                \draw (1) to[bend left] node[left] {} (2);
                \draw (1) to[bend right] node[above, yshift = 0.18cm] {$>$} (2);
                \draw (1) to node[left] {} (2);
                
        \end{tikzpicture} & \begin{tikzpicture}[auto, node distance=1.5cm, scale = 0.4, transform shape]
            \node[state] (1) {1};
            
            \node[state] (2) [right of = 1] {2};
            \node[state] (3) [above of = 2] {2};
            \node[state] (4) [below of = 2] {2};
                
            \node[state] (5) [right of = 2] {1};

                \draw (1) to node[left] {} (2);
                \draw (1) to node[left] {} (3);
                \draw (1) to node[left] {} (4);
                
                \draw (5) to node[left] {} (2);
                \draw (5) to node[left] {} (3);
                \draw (5) to node[left] {} (4);

        \end{tikzpicture}\\ \hline
    \end{tabular}
      \caption{The $F_4$ and $G_2$ exceptional Lie Algebras.}
      \label{fig:f4g2}
  \end{figure}
  \\We see again in the highest root pattern the relationship between the patterns for the original simply laced diagrams and those constructed via the trick. 
  \\This completes the list of diagrams for which the process of finding new valid patterns terminates, and thus completes the list of finite dimensional simple Lie algebras. The remaining algebras are all infinite dimensional, but there are still subtleties in how that behavior manifests.
\section{Affine Lie Algebras: Repeating Diagrams for Imaginary Roots}\label{sec:Affine}
 The second part of the trichotomy are the ``Affine Lie Algebras," with positive semidefinite Cartan Matrices, The first example is the $\tilde{A}_n$, which should be read as ``affine $A_n$". In this and the third part of the trichotomy, there appear ``imaginary roots," which have nonpositive norm according to the bilinear form $DA$. Real roots, which is to say those with positive norm, have corresponding root spaces of dimension one. Hence it makes sense to discuss ``the" root vector corresponding to a root, or at least a choice of basis subordinate to the root spaces obviates a need to make an additional choice of representative. The root spaces of imaginary roots, on the other hand, have dimension (multiplicity) which is greater than one. Furthermore, the linear dependencies among different constructions of root vectors do not lend themselves to a canonical choice of basis for each root space. The long-time behavior of root multiplicities along imaginary directions does depend on if the norm is negative or zero, see for example \cite{CarboneMultiplicities}. For now we specialize to the Affine Lie algebras, for which $DA$ is of positive semidefinite type. These have polynomial growth rate in the number of generators used. The easiest way to produce a semipositive definite generalized Cartan matrix is just to close the path of an $A_n$ to a loop. This includes, formally, as a subclass the loop on two elements, i.e. the graph $\tilde{A}_1$, which is a pair of vertices connected by two undirected lines.

\begin{figure}[h!]
    \centering
    \includegraphics[scale = 0.5]{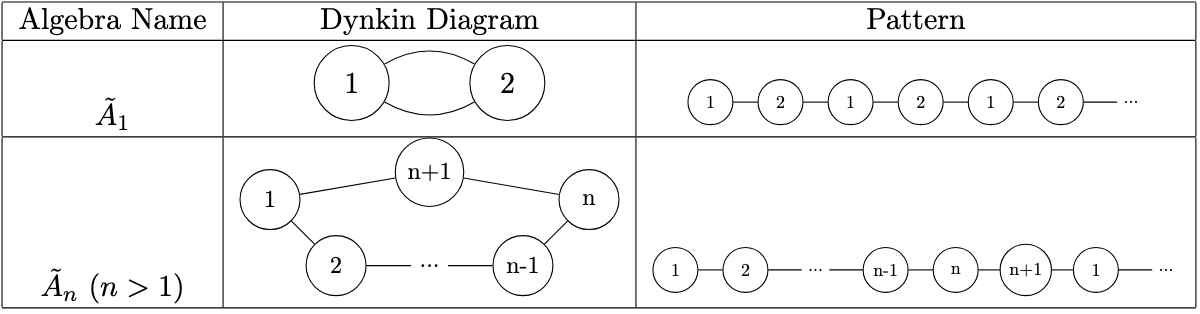}
    \caption{The Affine $A$ Algebras}
    \label{fig:AffineLoop}
\end{figure}

In these cases we see the two primary features particular to Affine algebras on clear display. The imaginary roots have presentations which are patterns that repeat some number of times, and may be extended indefinitely. We also see, for the first time, that we will not be able to deform some patterns with the same number of each generator in to some others. For example, if we take the algebra $\tilde{A}_2$, we could take the element given by the pattern:\\
\begin{center}
 \begin{tikzpicture}[auto, node distance=1.5cm, scale = 0.4, transform shape]
 \node[state] (1) {1};
 \node[state] (2) [right of = 1] {2};
 \node[state] (3) [right of = 2] {3};

 \draw (1) to (2);
 \draw (2) to (3);
 \end{tikzpicture}
\end{center}

which would be independent with either\\
\begin{center}
 \begin{tikzpicture}[auto, node distance=1.5cm, scale = 0.4, transform shape]
 \node[state] (1) {2};
 \node[state] (2) [right of = 1] {3};
 \node[state] (3) [right of = 2] {1};

 \draw (1) to (2);
 \draw (2) to (3);
 \end{tikzpicture}
\end{center}
or with\\
\begin{center}
 \begin{tikzpicture}[auto, node distance=1.5cm, scale = 0.4, transform shape]
 \node[state] (1) {1};
 \node[state] (2) [right of = 1] {3};
 \node[state] (3) [right of = 2] {2};

 \draw (1) to (2);
 \draw (2) to (3);
 \end{tikzpicture}
\end{center}
and the latter two would be independent of one another. However, there would be a linear dependence between all three, which is given by the Jacobi identity. This gives a root space multiplicity of two for the root given by one of each generator. It is a general phenomenon with Affine algebras (given by connected graphs) that the imaginary positive root spaces all have the same dimension. This is manifest in the construction by loop algebras in chapter seven of Kac's book \cite{Kac_2003}. Aside from graphs containing loops, there are a few ways to create infinitely extensible patterns. We now present the Dynkin diagrams and their associated infinitely extensible patterns.

\subsection{Two Special Features}
In the finite case, the presence of what we called a ``special feature," namely a double edge or a vertex of degree three, allowed us to bounce back and repeat our steps of adding vertices in reverse. The imaginary roots of these algebras are given by iterations of a repeating pattern using this bounce. To remain affine, the graphs with two special features must have them at each end of the path, and the features must be no more complicated than either of these options. If we order the vertices in the path, there are nine ways to create graphs with the features we want, but three are redundant when the orders are reversed. This leaves six families. The first three of these are untwisted, meaning that they can be obtained directly from simple algebras via central extensions of the loop algebras, but the remaining three require a more complicated construction involving twisting the loop algebra by automorphisms of the Dynkin graphs.
\begin{figure}[h]
    \centering
  \begin{tabular}{|c|c|c|}
    \hline
       Name & Diagram & Pattern\\\hline
       $\tilde{D}_n$ & \begin{tikzpicture}[auto, node distance=1.5cm, scale = 0.4, transform shape]

        \node[state] (1) {3};
        
        \node[state] (2) [below left of = 1] {1};
        \node[state] (3) [above left of = 1] {2};
        
        \node[state] (4) [right of = 1] {4};
        
        \node (5) [right of = 4] {...};
        
        \node[state] (6) [right of = 5] {n-1};

        \node[state] (7) [above right of = 6] {n+1};
        \node[state] (8) [below right of = 6] {n};

        \draw (1) to (2);
        \draw (1) to (3);
        \draw (1) to (4);

        \draw (4) to (5);
        
        \draw (5) to (6);

        \draw (6) to (7);
        \draw (6) to (8);
        \end{tikzpicture} &
        \begin{tikzpicture}[auto, node distance=1.5cm, scale = 0.4, transform shape]

        \node[state] (1) {3};
        
        \node[state] (2) [below right of = 1] {1};
        \node[state] (3) [above right of = 1] {2};
        
        \node[state] (4) [right of = 1] {4};
        
        \node (5) [right of = 4] {...};
        
        \node[state] (6) [right of = 5] {n-1};

        \node[state] (7) [above right of = 6] {n+1};
        \node[state] (8) [below right of = 6] {n};

        \node[state] (9) [below right of = 7] {n-1};

        \node (10) [right of = 9] {...};

        \node[state] (11) [right of = 10] {4};

        \node[state] (12) [right of = 11] {3};
        
        \node[state] (13) [above right of = 12] {2};
        \node[state] (14) [below right of = 12] {1};

        \node[state] (15) [above right of = 14] {3};

        \node (16) [right of = 15] {...};

        \draw (1) to (2);
        \draw (1) to (3);
        \draw (1) to (4);

        \draw (4) to (5);
        
        \draw (5) to (6);

        \draw (6) to (7);
        \draw (6) to (8);

        \draw (7) to (9);
        \draw (8) to (9);
        
        \draw (9) to (10);

        \draw (10) to (11);
        
        \draw (11) to (12);

        \draw (12) to (13);
        \draw (12) to (14);

        \draw (13) to (15);
        \draw (14) to (15);

        \draw (15) to (16);
        \end{tikzpicture}

       \\ \hline
       $\tilde{B}_n$ &\begin{tikzpicture}[auto, node distance=1.5cm, scale = 0.4, transform shape]

       \node[state] (a) {2};
       \node (b) [right of = a] {...};
       \node[state] (c)[right of = {b}] {n-1};

        \node[state] (1) [left of = a] {1};
        \node[state] (2)  [below right of = c] {n};
        \node[state] (3) [above right of = c] {n+1};
        
       \draw (a) to (b);
       \draw (b) to (c);
       \draw (1) to [bend left] (a);
       \draw (1) to [bend right] node[above] {$<$}(a) ;
       \draw (c) to (2);
       \draw (c) to (3);

        \end{tikzpicture} & \begin{tikzpicture}[auto, node distance=1.5cm, scale = 0.4, transform shape]

       \node[state] (a) {2};
       \node[state] (2one) [above right of = a] {1};
       \node[state] (2two) [right of = a] {3};
       \node[state] (FUBAR) [below right of =a] {1};
       \node (b) [right of = 2two] {...};
       \node[state] (c)[right of = b] {n-1};

        \draw (a) to (FUBAR);
       \node[state] (n+1) [below right of = c] {n};
       \node[state] (n+2) [above right of = c] {n+1};

       \draw (2two) to (b);
       \draw (a) to (2two);
       \draw (a) to (2one);
       \draw (b) to (c);

       \node[state] (d) [below right of = n+2] {n-1};
       \node (e) [right of = d] {...};
       \node[state] (f)[right of = e] {2};
       \node[state] (2three) [above right of = f] {1};
       \node[state] (2four) [below right of = f] {1};
       \node[state] (1two) [above right of = 2four] {2};
       \node (end) [right of = 1two] {...};

       \draw (d) to (e);
       \draw (e) to (f);
       \draw (c) to (n+1);
       \draw (c) to (n+2);
       \draw (n+2) to (d);
       \draw (n+1) to (d);

       \draw (f) to (2three);
       \draw (f) to (2four);
       \draw (2three) to (1two);
       \draw (2four) to (1two);

       \draw (1two) to (end);

        \end{tikzpicture} \\ \hline

       $\tilde{C}_n$ &\begin{tikzpicture}[auto, node distance=1.5cm, scale = 0.4, transform shape]

       \node[state] (a) {2};
       \node (b) [right of = a] {...};
       \node[state] (c)[right of = {b}] {n};

       \draw (a) to (b);
       \draw (b) to (c);

       \node[state] (1) [left of = a] {1};
       \draw (1) to [bend left] (a);
       \draw (1) to [bend right] node[above]{$>$} (a);

       \node[state] (2) [right of = c] {n+1};
       \draw (2) to [bend right] (c);
       \draw (2) to [bend left] node[above]{$<$} (c);

        \end{tikzpicture} & \begin{tikzpicture}[auto, node distance=1.5cm, scale = 0.4, transform shape]

       \node[state] (a) {1};
       \node (b) [right of = a] {...};
       \node[state] (c)[right of = b] {n};

       \draw (a) to (b);
       \draw (b) to (c);
         \node[state] (n+1) [right of = c] {n+1};
       \node[state] (d) [right of = n+1] {n};
       \node (e) [right of = d] {...};
       \node[state] (f)[right of = e] {1};
       \node[state] (g) [right of = f] {2};
       \node[state] (h) [right of = g] {...};

       \draw (d) to (e);
       \draw (e) to (f);
       \draw (c) to (n+1);
       \draw (n+1) to (d);
       \draw (f) to (g);
       \draw (g) to (h);

        \end{tikzpicture}  \\ \hline

        $A_{2n+1}^{(2)}$ &\begin{tikzpicture}[auto, node distance=1.5cm, scale = 0.4, transform shape]

       \node[state] (a) {2};
       \node (b) [right of = a] {...};
       \node[state] (c)[right of = {b}] {n-1};

        \node[state] (1) [left of = a] {1};
        \node[state] (2)  [below right of = c] {n};
        \node[state] (3) [above right of = c] {n+1};
        
       \draw (a) to (b);
       \draw (b) to (c);
       \draw (1) to [bend left] (a);
       \draw (1) to [bend right] node[above] {$>$}(a) ;
       \draw (c) to (2);
       \draw (c) to (3);

        \end{tikzpicture} & \begin{tikzpicture}[auto, node distance=1.5cm, scale = 0.4, transform shape]

       \node[state] (a) {2};
       \node[state] (2one) [above right of = a] {1};
       \node[state] (2two) [right of = a] {3};
       \node[state] (2tworeal) [below right of = a] {1};
       \draw (a) to (2tworeal);
     
       \node (b) [right of = 2two] {...};
       \node[state] (c)[right of = b] {n-1};

       \node[state] (n+1) [below right of = c] {n};
       \node[state] (n+2) [above right of = c] {n+1};

       \draw (2two) to (b);
       \draw (a) to (2two);
       \draw (a) to (2one);
       \draw (b) to (c);

       \node[state] (d) [below right of = n+2] {n-1};
       \node (e) [right of = d] {...};
       \node[state] (f)[right of = e] {2};
       \node[state] (2three) [above right of = f] {1};
       \node[state] (2four) [below right of = f] {1};
       \node[state] (1two) [above right of = 2four] {2};
       \node (end) [right of = 1two] {...};

       \draw (d) to (e);
       \draw (e) to (f);
       \draw (c) to (n+1);
       \draw (c) to (n+2);
       \draw (n+2) to (d);
       \draw (n+1) to (d);

       \draw (f) to (2three);
       \draw (f) to (2four);
       \draw (2three) to (1two);
       \draw (2four) to (1two);

       \draw (1two) to (end);

        \end{tikzpicture}  \\ \hline

        $A_{2n+2}^{(2)}$&\begin{tikzpicture}[auto, node distance=1.5cm, scale = 0.4, transform shape]

       \node[state] (a) {2};
       \node (b) [right of = a] {...};
       \node[state] (c)[right of = {b}] {n};

       \draw (a) to (b);
       \draw (b) to (c);

       \node[state] (1) [left of = a] {1};
       \draw (1) to [bend left] (a);
       \draw (1) to [bend right] node[above]{$>$} (a);

       \node[state] (2) [right of = c] {n+1};
       \draw (2) to [bend right] (c);
       \draw (2) to [bend left] node[above]{$>$} (c);

        \end{tikzpicture} & \begin{tikzpicture}[auto, node distance=1.5cm, scale = 0.4, transform shape]

\node[state] (a) {1};
       \node[state] (2one) [above right of = a] {2};
       \node[state] (2two) [right of = a] {2};
     
       \node (b) [right of = 2two] {...};
       \node[state] (c)[right of = b] {n-1};

       \node[state] (n+1) [below right of = c] {n};
       \node[state] (n+2) [above right of = c] {n+1};

       \draw (2two) to (b);
       \draw (a) to (2two);
       \draw (a) to (2one);
       \draw (b) to (c);

       \node[state] (d) [below right of = n+2] {n-1};
       \node (e) [right of = d] {...};
       \node[state] (f)[right of = e] {2};
       \node[state] (2three) [right of = f] {1};
       \node[state] (2four) [right of = 2three] {2};

       \node (end) [right of = 2four] {...};

       \draw (d) to (e);
       \draw (e) to (f);
       \draw (c) to (n+1);
       \draw (c) to (n+2);
       \draw (n+2) to (d);
       \draw (n+1) to (d);

       \draw (f) to (2three);
       \draw (2three) to (2four);

       \draw (2four) to (end);
     
        \end{tikzpicture}   \\ \hline

        $D_{n+2}^{(2)}$ &\begin{tikzpicture}[auto, node distance=1.5cm, scale = 0.4, transform shape]

       \node[state] (a) {2};
       \node (b) [right of = a] {...};
       \node[state] (c)[right of = {b}] {n};

       \draw (a) to (b);
       \draw (b) to (c);

       \node[state] (1) [left of = a] {1};
       \draw (1) to [bend left] (a);
       \draw (1) to [bend right] node[above]{$<$} (a);

       \node[state] (2) [right of = c] {n+1};
       \draw (2) to [bend right] (c);
       \draw (2) to [bend left] node[above]{$>$} (c);

        \end{tikzpicture} & \begin{tikzpicture}[auto, node distance=1.5cm, scale = 0.4, transform shape]
        \node[state] (a) {2};
       \node[state] (2one) [above right of = a] {1};
       \node[state] (2two) [right of = a] {3};
       \node[state] (2tworeal) [below right of = a] {1};
       \draw (a) to (2tworeal);
     
       \node (b) [right of = 2two] {...};
       \node[state] (c)[right of = b] {n-1};

       \node[state] (n+1) [below right of = c] {n};
       \node[state] (n+2) [above right of = c] {n+1};

       \draw (2two) to (b);
       \draw (a) to (2two);
       \draw (a) to (2one);
       \draw (b) to (c);

       \node[state] (d) [below right of = n+2] {n-1};
       \node (e) [right of = d] {...};
       \node[state] (f)[right of = e] {2};
       \node[state] (2three) [above right of = f] {1};
       \node[state] (2four) [below right of = f] {1};
       \node[state] (1two) [above right of = 2four] {2};
       \node (end) [right of = 1two] {...};

       \draw (d) to (e);
       \draw (e) to (f);
       \draw (c) to (n+1);
       \draw (c) to (n+2);
       \draw (n+2) to (d);
       \draw (n+1) to (d);

       \draw (f) to (2three);
       \draw (f) to (2four);
       \draw (2three) to (1two);
       \draw (2four) to (1two);

       \draw (1two) to (end);

        \end{tikzpicture}  \\ \hline
    \end{tabular}
    \caption{The Affine Lie Algebras With Two Special Features}
    \label{fig:twospecial}
\end{figure}
The type $\tilde{D}$ graphs can be used to obtain the type $\tilde{B}$ and $D^{(2)}$ graphs by applying Kac's trick via the $\Z_2$ symmetry on one pair or both pairs of leaves respectively. The bilateral $\Z_2$ symmetry of the $\tilde{D}$ type graphs can be used to obtain the $A^{(2)}_{odd}$ graphs, and a further application via the $\Z_2$ symmetry on the leaves gives the $A^{2}_{even}$ type. $\tilde{C}$ graphs have a bilateral symmetry which allows for the construction of smaller $\tilde{C}$ graphs. $\tilde{C}$ graphs can also be constructed by applying bilateral symmetry to an $\tilde{A}$ of double the size.
\newpage
\subsection{$\tilde{E}$ Algebras}
There are also affine analogues to the $E_n$ family, which may be obtained from the normal $E_n$ via affinization or extension of the loop algebra. Each addition of a vertex allows for the creation of a building block which can be iterated infinitely and which are each clearly related to the highest roots for the non-affine $E_n$ algebras. All of the indefinitely iterable patterns enjoy a bilateral symmetry within each segment, which only $E_6$ did in the non-affine case.

\begin{figure}[h]
    \centering
        \begin{tabular}{|c|c|c|}
    \hline
        Algebra Name & Dynkin Diagram & Pattern \\ \hline
        $\tilde{E}_6$  & \begin{tikzpicture}[auto, node distance=2cm, scale = 0.3, transform shape]
            \node[state] (1) {1};
            \node[state] (2) [right of = 1] {2};
            \node[state] (3) [right of = 2] {3};
            \node[state] (x) [above of = 3] {x};
            \node[state] (y) [above of = x] {y};
            \node[state] (4) [right of = 3] {4};
            \node[state] (5) [right of = 4] {5};
            \draw (1) to node[left] {} (2);
            \draw (2) to node[left] {} (3);
            \draw (3) to node[left] {} (4);
            \draw (4) to node[left] {} (5);
            \draw (3) to node[left] {} (x);
            \draw (x) to (y);

        \end{tikzpicture} & \begin{tikzpicture}[auto, node distance=1.7cm, scale = 0.3, transform shape]
\node[state] (31) {3};

\node[state] (x2) [right of = 31] {x};
\node[state] (42) [above of = x2] {4};
\node[state] (22) [below of = x2] {2};

\node[state] (53) [right of = 42] {5};
\node[state] (y3) [right of = x2] {y};
\node[state] (13) [right of = 22] {1};
\node[state] (33) [below of = 13] {3};

\node[state] (44) [right of = 53] {4};
\node[state] (x4) [right of = y3] {x};
\node[state] (24) [right of = 13] {2};

\node[state] (35) [right of = x4] {3};


\draw (31) to (x2);
\draw (31) to (42);
\draw (31) to (22);

\draw (42) to (53);
\draw (x2) to (y3);
\draw (22) to (13);
\draw (42) to (33);
\draw (x2) to (33);
\draw (22) to (33);

\draw (44) to (53);
\draw (x4) to (y3);
\draw (24) to (13);
\draw (44) to (33);
\draw (x4) to (33);
\draw (24) to (33);

\draw (35) to (x4);
\draw (35) to (44);
\draw (35) to (24);

\end{tikzpicture}\\ \hline
        $\tilde{E}_7$  & \begin{tikzpicture}[auto, node distance=2cm, scale = 0.3, transform shape]
            \node[state] (1) {1};
            \node[state] (2) [right of = 1] {2};
            \node[state] (3) [right of = 2] {3};
            \node[state] (4) [right of = 3] {4};
            \node[state] (5) [right of = 4] {5};
            \node[state] (x) [above of = 4] {x};
            \node[state] (6) [right of = 5] {6};
            \node[state] (7) [right of = 6] {7};
            \draw (1) to node[left] {} (2);
            \draw (2) to node[left] {} (3);
            \draw (3) to node[left] {} (4);
            \draw (4) to node[left] {} (5);
            \draw (4) to node[left] {} (x);
            \draw (5) to node[left] {} (6);
            \draw (6) to (7);

        \end{tikzpicture} & \begin{tikzpicture}[auto, node distance=1.7cm, scale = 0.3, transform shape]
\node[state] (41) {4};

\node[state] (52) [above right of = 41] {5};
\node[state] (x2) [right of = 41, xshift =-0.5cm] {x};
\node[state] (32) [below right of = 41] {3};

\node[state] (63) [above right of  = 52 ] {6};
\node[state] (43) [below right of = 52 ] {4};
\node[state] (23) [below right of = 32] {2};

\node[state] (74) [above right of  =  63] {7};
\node[state] (54) [below right of  = 63 ] {5};
\node[state] (34) [below right of = 43 ] {3};
\node[state] (14) [ below right of = 23 ] {1};

\node[state] (65) [below right of = 74] {6};
\node[state] (45) [below right of = 54] {4};
\node[state] (25) [below right of = 34] {2};

\node[state] (56) [below right of = 65] {5};
\node[state] (x6) [right of = 45, xshift = -0.5cm] {x};
\node[state] (36) [below right of = 45 ] {3};

\node[state] (47) [below right of = 56] {4};


\draw (41) to (52);
\draw (41) to (x2);
\draw (41) to (32);

\draw (52) to (63);
\draw (52) to (43);
\draw (x2) to (43);
\draw (32) to (43);
\draw (32) to (23);

\draw (63) to (74);
\draw (63) to (54);
\draw (43) to (54);
\draw (43) to (34);
\draw (23) to (34);
\draw (23) to (14);

\draw (74) to (65);
\draw (54) to (65);
\draw (54) to (45);
\draw (34) to (45);
\draw (34) to (25);
\draw (14) to (25);

\draw (65) to (56);
\draw (45) to (56);
\draw (45) to (x6);
\draw (45) to (36);
\draw (25) to (36);

\draw (56) to (47);
\draw (x6) to (47);
\draw (36) to (47);

\end{tikzpicture}\\ \hline
        $\tilde{E}_8$  & \begin{tikzpicture}[auto, node distance=2cm, scale = 0.3, transform shape]
            \node[state] (0) {0};
            \node[state] (1) [right of = 0] {1};
            \node[state] (2) [right of = 1] {2};
            \node[state] (3) [right of = 2] {3};
            \node[state] (4) [right of = 3] {4};
            \node[state] (5) [right of = 4] {5};
            \node[state] (6) [right of = 5] {6};
            \node[state] (x) [above of = 5] {x};
            \node[state] (7) [right of = 6] {7};
            \draw (0) to (1);
            \draw (1) to node[left] {} (2);
            \draw (2) to node[left] {} (3);
            \draw (3) to node[left] {} (4);
            \draw (4) to node[left] {} (5);
            \draw (5) to node[left] {} (x);
            \draw (5) to node[left] {} (6);
            \draw (6) to node[left] {} (7);

        \end{tikzpicture} & \begin{tikzpicture}[auto, node distance=1.7cm, scale = 0.2, transform shape]
\node[state] (61) {6};

\node[state] (72) [above right of = 61] {7};
\node[state] (x2) [right of = 61, xshift = -0.5cm] {x};
\node[state] (52) [below right of = 61] {5};

\node[state] (83) [above right of = 72] {8};
\node[state] (63) [above right of = 52] {6};
\node[state] (43) [below right of = 52] {4};

\node[state] (74) [below right of = 83] {7};
\node[state] (54) [below right of = 63] {5};
\node[state] (34) [below right of = 43] {3};

\node[state] (65) [below right of = 74] {6};
\node[state] (45) [below right of = 54] {4};
\node[state] (25) [below right of = 34] {2};

\node[state] (x6) [right of = 65, xshift = -0.5cm] {x};
\node[state] (56) [below right of = 65] {5};
\node[state] (36) [below right of = 45] {3};
\node[state] (16) [below right of = 25] {1};

\node[state] (67) [above right of = 56] {6};
\node[state] (47) [above right of = 36] {4};
\node[state] (27) [above right of = 16] {2};

\node[state] (78) [above right of = 67] {7};
\node[state] (58) [above right of = 47] {5};
\node[state] (38) [above right of = 27] {3};

\node[state] (89) [above right of = 78] {8};
\node[state] (69) [above right of = 58] {6};
\node[state] (49) [above right of = 38] {4};

\node[state] (710) [below right of = 89] {7};
\node[state] (x10) [right of = 69, xshift = -0.5cm] {x};
\node[state] (510) [below right of = 69] {5};

\node[state] (611) [above right of = 510] {6};


\draw (61) to (72);
\draw (61) to (x2);
\draw (61) to (52);

\draw (72) to (83);
\draw (72) to (63);
\draw (x2) to (63);
\draw (52) to (63);
\draw (52) to (43);

\draw (83) to (74);
\draw (63) to (74);
\draw (63) to (54);
\draw (43) to (54);
\draw (43) to (34);

\draw (74) to (65);
\draw (54) to (65);
\draw (54) to (45);
\draw (34) to (45);
\draw (34) to (25);

\draw (65) to (x6);
\draw (65) to (56);
\draw (45) to (56);
\draw (45) to (36);
\draw (25) to (36);
\draw (25) to (16);

\draw (67) to (x6);
\draw (67) to (56);
\draw (47) to (56);
\draw (47) to (36);
\draw (27) to (36);
\draw (27) to (16);

\draw (611) to (710);
\draw (611) to (x10);
\draw (611) to (510);

\draw (710) to (89);
\draw (710) to (69);
\draw (x10) to (69);
\draw (510) to (69);
\draw (510) to (49);

\draw (89) to (78);
\draw (69) to (78);
\draw (69) to (58);
\draw (49) to (58);
\draw (49) to (38);

\draw (78) to (67);
\draw (58) to (67);
\draw (58) to (47);
\draw (38) to (47);
\draw (38) to (27);

\end{tikzpicture}
        \\ 
        \hline
    \end{tabular}
    
    \caption{The Affine $\tilde{E}_n$ Algebras}
    \label{fig:affineE}
\end{figure}
\newpage
\subsection{$G_2$ and $F_4$-Like Affine Diagrams}

The remaining Affine algebras have Dynkin diagrams which are extensions of the graphs $G_2$ and $F_4$, in the sense that they contain these graphs as oriented subgraphs. Each of these has a single oriented multiple edge, and therefore can be obtained from a simply laced graph via Kac's trick. We begin with the algebras with graphs extending that of $G_2$. There are three such graphs that are affine. The affinization $\tilde{G}_2$ of $G_2$ and the twisted algebra $A_2^{(2)}$ are both the result of applying Kac's trick to $\tilde{D^4}$, either via a $\Z_3$ symmetry on three of the leaves or the $\Z_4$ symmetry on all four. $D_4^{(3)}$, on the other hand, is obtained via the $\Z_3$ symmetry of the $\tilde{E}_6$ graph. This method of obtaining the algebras is manifest in the choice of indefinitely iterable pattern for each of these, as it is clear that relabeling according to the symmetry will give one from the other.

\begin{figure}[h!]
    \centering
    \begin{tabular}{|c|c|c|}
        \hline
        Name & Diagram  & Pattern\\ \hline

         $\tilde{G}_2$ & 
         
         \begin{tikzpicture}[auto, node distance=1.5cm, scale = 0.4, transform shape]
             \node[state] (1) {1};
             \node[state] (2) [right of = 1] {2};
             \node[state] (3) [right of = 2] {3};

             \draw (2) to [bend right] node [above = 2.1] {$\mathbb{>}$} (3);
             \draw (2) to [bend left] (3);
             \draw (2) to  (3);

             \draw (1) to (2);
         \end{tikzpicture}

         &

         \begin{tikzpicture}[auto, node distance=1.5cm, scale = 0.3, transform shape]
             \node[state] (1) {2};
             \node[state] (2) [right of = 1] {3};
             \node[state] (3) [below of = 2] {3};
             \node[state] (4) [below of = 3] {3};
             \node[state] (5) [above of = 2] {1};
             \node[state] (6) [right of = 2 ] {2};

             \draw (1) to (2);
             \draw (1) to (3);
             \draw (1) to (4);
             \draw (1) to (5);

             \draw (2) to (6);
             \draw (3) to (6);
             \draw (4) to (6);
             \draw (5) to (6);

         \end{tikzpicture}

         \\ \hline

         $D_4^{(3)}$ & 
         
         \begin{tikzpicture}[auto, node distance=1.5cm, scale = 0.5, transform shape]
             \node[state] (1) {1};
             \node[state] (2) [right of = 1] {2};
             \node[state] (3) [right of = 2] {3};

             \draw (1) to [bend right] node [above = 2.1] {$\mathbb{>}$} (2);
             \draw (1) to [bend left] (2);
             \draw (1) to  (2);

             \draw (2) to (3);
         \end{tikzpicture}

         &

         \begin{tikzpicture}[auto, node distance=1.7cm, scale = 0.3, transform shape]
\node[state] (31) {1};

\node[state] (x2) [right of = 31] {2};
\node[state] (42) [above of = x2] {2};
\node[state] (22) [below of = x2] {2};

\node[state] (53) [right of = 42] {3};
\node[state] (y3) [right of = x2] {3};
\node[state] (13) [right of = 22] {3};
\node[state] (33) [below of = 13] {1};

\node[state] (44) [right of = 53] {2};
\node[state] (x4) [right of = y3] {2};
\node[state] (24) [right of = 13] {2};

\node[state] (35) [right of = x4] {1};


\draw (31) to (x2);
\draw (31) to (42);
\draw (31) to (22);

\draw (42) to (53);
\draw (42) to (y3);
\draw (42) to (13);
\draw (x2) to (53);
\draw (x2) to (y3);
\draw (x2) to (13);
\draw (22) to (53);
\draw (22) to (y3);
\draw (22) to (13);
\draw (42) to (33);
\draw (x2) to (33);
\draw (22) to (33);

\draw (44) to (53);
\draw (44) to (y3);
\draw (44) to (13);
\draw (x4) to (53);
\draw (x4) to (y3);
\draw (x4) to (13);
\draw (24) to (53);
\draw (24) to (y3);
\draw (24) to (13);
\draw (44) to (33);
\draw (x4) to (33);
\draw (24) to (33);

\draw (35) to (x4);
\draw (35) to (44);
\draw (35) to (24);

\end{tikzpicture}

         \\ \hline

         $A_2^{(2)}$ & 
         
         \begin{tikzpicture}[auto, node distance=1.5cm, scale = 0.5, transform shape]
             \node[state] (1) {1};
             \node[state] (2) [right of = 1] {2};

            \draw (1) to [bend left = 60] (2);
            \draw (1) to [bend left = 30] (2);
            \draw (1) to [bend right = 30] node[above, yshift = 0.1cm] {$>$} (2);
            \draw (1) to [bend right = 60] (2);
         \end{tikzpicture}

         &

         \begin{tikzpicture}[auto, node distance=1.5cm, scale = 0.3, transform shape]
             \node[state] (1) {1};
             
             \node[state] (2) [right of = 1] {2};
             \node[state] (3) [below of = 2] {2};
             \node[state] (4) [below of = 3] {2};
             \node[state] (5) [above of = 2] {2};
             
             \node[state] (6) [right of = 2 ] {1};

             \draw (1) to (2);
             \draw (1) to (3);
             \draw (1) to (4);
             \draw (1) to (5);

             \draw (2) to (6);
             \draw (3) to (6);
             \draw (4) to (6);
             \draw (5) to (6);

         \end{tikzpicture}

         \\ \hline
    
    \end{tabular}
    \caption{$G_2$ like Affine Lie Algebras}
    \label{fig:G2Likes}
\end{figure}
\newpage
Finally we have the two $F_4$ like Affine algebras. These are obtained via extending the $F_4$ graph by adding an additional vertex to one side or the other. $\tilde{F}_4$ is the affinization of $F_4$, but is also obtained from $\tilde{E}_6$ via a bilateral symmetry, just as $F_4$ is obtained from $E_6$. $E_6^{(2)}$, on the other hand, is obtained from the bilateral symmetry of $\tilde{E}_7$. Again, we see that this method for obtaining the algebras is manifest in the indefinitely iterable pattern chosen.
\newpage
\begin{figure}[h!]
    \centering
    \begin{tabular}{|c|c|c|} \hline
        Name & Diagram & Pattern\\ \hline
        $\tilde{F}_4$ &

        \begin{tikzpicture}[auto, node distance=1.5cm, scale = 0.4, transform shape]
        \node[state] (1) {1};
        \node[state] (2) [right of = 1] {2};
        \node[state] (3) [right of = 2] {3};
        \node[state] (4) [right of = 3] {4};
        \node[state] (5) [right of = 4] {5};

        \draw (1) to (2);
        \draw (2) to (3);

        \draw (3) to [bend left] (4);
        \draw (3) to [bend right] node[above, yshift = 0.1cm] {$>$} (4);
        \draw (4) to (5);
        
        \end{tikzpicture}
        & 
     
\begin{tikzpicture}[auto, node distance=1.7cm, scale = 0.3, transform shape]
\node[state] (31) {3};

\node[state] (x2) [right of = 31] {4};
\node[state] (42) [above of = x2] {2};
\node[state] (22) [below of = x2] {4};

\node[state] (53) [right of = 42] {1};
\node[state] (y3) [right of = x2] {5};
\node[state] (13) [right of = 22] {5};
\node[state] (33) [below of = 13] {3};

\node[state] (44) [right of = 53] {2};
\node[state] (x4) [right of = y3] {4};
\node[state] (24) [right of = 13] {4};

\draw (x2) to (13);
\draw (22) to (y3);
\draw (x4) to (13);
\draw (24) to (y3);

\node[state] (35) [right of = x4] {3};


\draw (31) to (x2);
\draw (31) to (42);
\draw (31) to (22);

\draw (42) to (53);
\draw (x2) to (y3);
\draw (22) to (13);
\draw (42) to (33);
\draw (x2) to (33);
\draw (22) to (33);

\draw (44) to (53);
\draw (x4) to (y3);
\draw (24) to (13);
\draw (44) to (33);
\draw (x4) to (33);
\draw (24) to (33);

\draw (35) to (x4);
\draw (35) to (44);
\draw (35) to (24);

        \end{tikzpicture}
        \\
        \hline

        $E_6^{(2)}$ &

        \begin{tikzpicture}[auto, node distance=1.5cm, scale = 0.4, transform shape]
        \node[state] (1) {1};
        \node[state] (2) [right of = 1] {2};
        \node[state] (3) [right of = 2] {3};
        \node[state] (4) [right of = 3] {4};
        \node[state] (5) [right of = 4] {5};

        \draw (1) to (2);
        \draw (2) to [bend left] (3);
        \draw (2) to [bend right] node[above, yshift = 0.1cm] {$>$} (3);
        \draw (3) to (4);
        \draw (4) to (5);
        \end{tikzpicture}
        & 
        \begin{tikzpicture}[auto, node distance=1.7cm, scale = 0.3, transform shape]
\node[state] (41) {2};

\node[state] (52) [above right of = 41] {3};
\node[state] (x2) [right of = 41, xshift =-0.5cm] {1};
\node[state] (32) [below right of = 41] {3};

\node[state] (63) [above right of  = 52 ] {4};
\node[state] (43) [below right of = 52 ] {2};
\node[state] (23) [below right of = 32] {4};
\draw (52) to (23);
\draw (32) to (63);

\node[state] (74) [above right of  =  63] {5};
\node[state] (54) [below right of  = 63 ] {3};
\node[state] (34) [below right of = 43 ] {3};
\node[state] (14) [ below right of = 23 ] {5};
\draw (23) to (74);
\draw (23) to (54);
\draw (63) to (14);
\draw (63) to (34);

\node[state] (65) [below right of = 74] {4};
\node[state] (45) [below right of = 54] {2};
\node[state] (25) [below right of = 34] {4};
\draw (25) to (74);
\draw (25) to (54);
\draw (65) to (14);
\draw (65) to (34);

\node[state] (56) [below right of = 65] {3};
\node[state] (x6) [right of = 45, xshift = -0.5cm] {1};
\node[state] (36) [below right of = 45 ] {3};
\draw (56) to (25);
\draw (36) to (65);

\node[state] (47) [below right of = 56] {2};


\draw (41) to (52);
\draw (41) to (x2);
\draw (41) to (32);

\draw (52) to (63);
\draw (52) to (43);
\draw (x2) to (43);
\draw (32) to (43);
\draw (32) to (23);

\draw (63) to (74);
\draw (63) to (54);
\draw (43) to (54);
\draw (43) to (34);
\draw (23) to (34);
\draw (23) to (14);

\draw (74) to (65);
\draw (54) to (65);
\draw (54) to (45);
\draw (34) to (45);
\draw (34) to (25);
\draw (14) to (25);

\draw (65) to (56);
\draw (45) to (56);
\draw (45) to (x6);
\draw (45) to (36);
\draw (25) to (36);

\draw (56) to (47);
\draw (x6) to (47);
\draw (36) to (47);

\end{tikzpicture}
        \\
        \hline
    \end{tabular}
    \caption{$F_4$ like Affine Lie Algebras}
    \label{fig:F4likes}
\end{figure}
Cumulatively, this is an exhaustive list of the Affine Kac-Moody algebras.

\section{Indefinite (Wild Type) Lie Algebras: Exponential Growth Via Binary Coding}\label{sec:WildBinary}
We identify Dynkin diagrams up to isomorphism with generalized Cartan matrices up to permutation of indices. We define a pre-order on the Cartan matrices as follows:
\begin{dfn}\label{def:orderonmatrices}
If $A$ is an $n\times n$ generalized Cartan matrix, $B$ is a $k\times k$ generalized Cartan matrix, and $n\leq k$, we will write
$$ A \leq_c B$$
If there is an inclusion $\sigma:\{1,...,n\} \rightarrow \{1,...,k\}$ and 
$$ -A_{ij} \leq -B_{\sigma(i), \sigma(j)}$$
For all $i$ and $j$.
\end{dfn}

After the ensuing quotient is taken, this is equivalent to the order on Dynkin diagrams given by inclusion of directed multigraphs. We will need the following well known fact

\begin{prop}\label{prop:affinesareasplit}
    The irreducible Affine type generalized Cartan matrices are pairwise $\leq_c$ incomparable, and every generalized Cartan matrix is either $\leq_c$-less than some affine matrix, or $\leq_c$ greater than some affine matrix.
\end{prop}

and another, less immediately obvious one:

\begin{prop}\label{prop:orderisalgebraic}
    Let $A \leq_c B$, then any nonzero pattern for $A$ is also a nonzero pattern for $B$, where all indices are replaced according to the $\sigma$ used in the definition of $\leq_c$.
\end{prop}
\begin{proof}
    Since $ad_a^n(b) = 0$ implies $ad_a^{n+k}(b) = 0$ for all $k > 0$, the positive algebra $E_A$ associated to the generalized Cartan matrix $A$ is a quotient of the positive algebra $E_B$ associated with $B$. Take $f$ to be the quotient map. Therefore if we have a sequence $Ad_{x_1},...,Ad_{x_n}(a)$, where each of the $x_i$ and $a$ are all generators, and we take $y_1,..,y_n$, $b$, the corresponding generators in $B$, $f(Ad_{y_1},...,Ad_{y_n}(b)) = Ad_{x_1},...,Ad_{x_n}(a)$, and thus if the latter is nonzero, so is the former.
\end{proof}

If $A$ is an irreducible affine Cartan matrix, and $B\leq_c A$ but $B\neq A$, then $B$ is a direct sum of distinct Dynkin diagrams corresponding to simple Lie algebras. Thus it has finitely many nonzero root patterns and finitely many roots. If $A$ is affine, then attempting to produce new roots of $A$ will lead to a repeating pattern, such as one of those given in the tables above and depending only on the starting point. This gives a polynomial upper bound, indeed a linear upper bound on the number of roots using below a certain number of generators.

Suppose now that $A$ is an irreducible affine Cartan matrix $A\leq_c B$ and $B$ is an irreducible generalized Cartan matrix. If $Q$ is the iterable pattern for $A$, we may take $R$ to be the same pattern replacing the generators with those for $B$. If $B$ has an additional vertex, the corresponding generator may be appended to $R$ to produce a pattern we will call $S_1$, and we will call $R$ $S_0$. If $B$ contains only an additional (directed) edge, we may append to $R$ an additional copy of the generator corresponding to the target of the new (directed) edge, producing $R'$. If we set $S_0 = RRR$ and $S_1 = RR'R$, we may encode binary sequences $(i_1,i_2,...,i_n)$ into the algebra itself as $S_{i_n},...,S_{i_2}S_{i_1} a$ (where $a$ is the starting pattern). The set of such codes will be linearly independent for a given length. Since these patterns all use less than $Cn$ generators (for some constant $C$ depending on the algebra), there will be at least $2^{k/C}$ elements of length $k$ asymptotically. This may be seen in a very elementary example, beginning with the algebra $\tilde{A}_1$ and some algebras built upon it, namely:
\begin{figure}[h]
    \centering

        \begin{tabular}{c c}
        $\tilde{A}_1$ &      \begin{tikzpicture}[auto, node distance = 1.5 cm]
        \node[state] (1) {1};
        \node[state] (2) [right of = 1] {2};
        \draw (1) to[bend left] node[left] {} (2);
        \draw (1) to[bend right] node[left] {} (2);
    \end{tikzpicture}\\

        Added edge & \begin{tikzpicture}[auto, node distance=1.5 cm, ]
            \node[state] (1) {1};
            \node[state] (2) [right of = 1] {2};

            \draw[<->] (1) to [bend left = 30] (2);
            \draw[<->] (1) to [bend left = 60] (2);
            \draw[->] (1) to [bend right = 45] (2);
        \end{tikzpicture} \\
         Added vertex &
         \begin{tikzpicture}[auto, node distance=1.5 cm, scale = 0.7, transform shape]
             \node[state] (x) {x};
             \node[state] (1) [right of = 1] {1};
             \node[state] (2) [right of = 2] {2};
             \draw[-] (x) to (1);
             \draw[-] (1) to [bend left] (2);
             \draw[-] (1) to [bend right] (2);

         \end{tikzpicture} 
    \end{tabular}\\ 
    \caption{$\tilde{A}_1$ and two indefinite algebras Built Upon It}
    \label{fig:A2likes}
\end{figure}
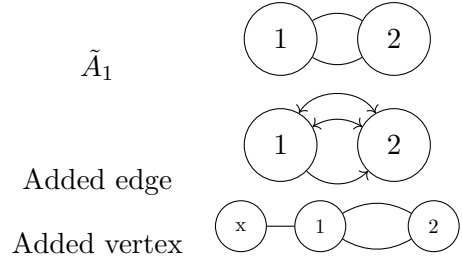
we will be able to encode binary sequences as follows:
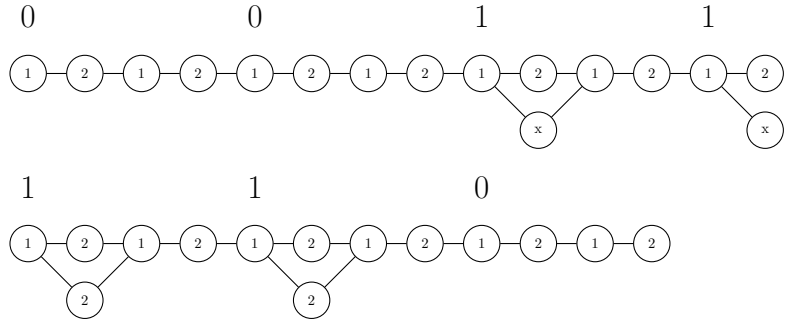
\begin{figure}[h]
    \centering
        \begin{tikzpicture}[auto, node distance=1.5 cm, scale = 0.5, transform shape]
    \node[] (1) {\Huge{0}};
    \node[] (1b) [right of = 1] {};
    \node[] (2) [right of = 1b] {};
    \node[] (2b) [right of = 2] {};
    \node[] (3) [right of = 2b] {\Huge{0}};
    \node[] (3b) [right of = 3] {};
    \node[] (4) [right of = 3b] {};
    \node[] (4b) [right of = 4] {};
    \node[] (5) [right of = 4b] {\Huge{1}};
    \node[] (5b) [right of = 5] {};
    \node[] (6) [right of = 5b] {};
    \node[] (6b) [right of = 6] {};
    \node[] (7) [right of = 6b] {\Huge{1}};
    \node[] (7b) [right of = 7] {};

    \node[state] (a) [below of = 1] {1};
    \node[state] (b) [below of = 1b] {2};
    \node[state] (c) [below of = 2] {1};
    \node[state] (d) [below of = 2b] {2};
    \node[state] (e) [below of = 3] {1};
    \node[state] (f) [below of = 3b] {2};
    \node[state] (g) [below of = 4] {1};
    \node[state] (h) [below of = 4b] {2};
    \node[state] (i) [below of = 5] {1};
    \node[state] (j) [below of = 5b] {2};
    \node[state] (k) [below of = 6] {1};
    \node[state] (l) [below of = 6b] {2};
    \node[state] (m) [below of = 7] {1};
    \node[state] (n) [below of = 7b] {2};

    \node[state] (exj) [below of = j] {x};
  
    \node[state] (exn) [below of = n] {x};

    \node (link2) [below of = a] {};
    \node (link) [below of = link2] {\Huge{1}};

    \node[state] (xa) [below of = link] {1};
    \node[state] (xb) [right of = xa] {2};
    \node[state] (xc) [right of = xb] {1};
    \node[state] (xd) [right of = xc] {2};
    \node[state] (xe) [right of = xd] {1};
    \node[state] (xf) [right of = xe] {2};
    \node[state] (xg) [right of = xf] {1};
    \node[state] (xh) [right of = xg] {2};
    \node[state] (xi) [right of = xh] {1};
    \node[state] (xj) [right of = xi] {2};
    \node[state] (xk) [right of = xj] {1};
    \node[state] (xl) [right of = xk] {2};

    \node (R) [above of = xe] {\Huge{1}};
    \node (Q) [above of = xi] {\Huge{0}};

    \node[state] (lexb) [below of = xb] {2};
    \node[state] (lexf) [below of = xf] {2};

    \draw (a) to (b);
    \draw (b) to (c);
    \draw (c) to (d);
    \draw (d) to (e);
    \draw (e) to (f);
    \draw (f) to (g);
    \draw (g) to (h);
    \draw (h) to (i);
    \draw (i) to (j);
    \draw (j) to (k);
    \draw (k) to (l);
    \draw (l) to (m);
    \draw (m) to (n);

    \draw (xa) to (xb);
    \draw (xb) to (xc);
    \draw (xc) to (xd);
    \draw (xd) to (xe);
    \draw (xe) to (xf);
    \draw (xf) to (xg);
    \draw (xg) to (xh);
    \draw (xh) to (xi);
    \draw (xi) to (xj);
    \draw (xj) to (xk);
    \draw (xk) to (xl);

    \draw (i) to (exj);
    \draw (exj) to (k);
    \draw (m) to (exn);

    \draw (xa) to (lexb);
    \draw (lexb) to (xc);
    \draw (xe) to (lexf);
    \draw (lexf) to (xg);
    \end{tikzpicture}
    \caption{Corresponding Encodings of Binary Sequences}
    \label{fig:encodings}
\end{figure}
\newpage
\section{Quiver Representations for Finite Type Roots}\label{sec:finitequiver}
\subsection{$A_n$ Quivers}
The simplest case, as always, is the $A_n$ case. Indecomposable representations correspond to intervals of $1 ,2,3,...,n$ with the usual order. Given such a sequence $i, i+1,...,i+j$ and an orientation of each edge, set $V_m = \C$ with basis element $e_m$, and set $f_{\{m,m+1\}}e_m = e_{m+1}$ or $f_{\{m,m+1\}}e_{m+1} = e_{m}$ as appropriate to the choice of orientation. Such a representation is clearly indecomposable, as any nontrivial decomposition would have to put some $e_m$ and its image in unconnected subspaces. This corresponds precisely to the positive roots we find for the $A_n$ case, which can all be represented in our scheme by the subpaths of $A_n$.
\subsection{$D_n$ Quivers}
The indecomposable representations of the $D_n$ quiver correspond to specific fragments of the maximal root diagram for $D_n$. Namely, any connected fragment of the $D_n$ diagram itself, using coefficients of one as before, but if more than one vertex $n-2$ is present, both the $n-1$ and $n$ vertices must be as well (and the maps from each to the vector space generated by the $n-2$ vertices must be linearly independent as maps from $\C$ to that space), furthermore one `tail,' i.e. the connected component of the subgraph on $1,...,n-2$ for each instance of $n-2$, must be longer than the other. This requirement is manifest in the roots in that we require the initial segment to be longer than the final for the root to be nonzero.

\subsection{$E_6$, $E_7$, and $E_8$} 
The $E_n$ case is complicated enough that we are now going to affirmatively fix an orientation for each case, which will be shown in the diagrams, and further we will only diagram the quiver representations corresponding to the highest roots.

\subsubsection{$E_6$}
The highest root of $E_6$ has a representation which looks like
\begin{center}
\begin{tikzcd}
    \C \arrow[r, "h_2"] & \C^{2} \arrow[r, "h_1"] & \C^{3} \arrow[r, "g_1"] \arrow[d, "f_1"] & \C^{2} \arrow[r, "g_2"] &\C\\
   &  & \C^{2}  & & \\
\end{tikzcd}
\end{center}
generically, up to a choice of basis this is isomorphic to:
\begin{center}
\begin{tikzcd}
    \C \arrow[r, "Id\oplus 0"] & \C\oplus\C \arrow[r, "Id\oplus 0"] & \C\oplus\C\oplus\C \arrow[to =  1-6, "(\pi_1 + \pi_3)\oplus (\pi_2 + \pi_3)"] \arrow[d, "\pi_1\oplus \pi_2"] & & & \C\oplus \C \arrow[r, "\pi_1"] &\C\\
   &  & \C\oplus \C  & & \\
\end{tikzcd}
\end{center}
\newpage
\subsubsection{$E_7$}
The highest root of $E_7$ has representation:
\begin{center}
    \begin{tikzcd}
        \C \arrow[d, "f_1"] \\
        \C^{2} \arrow[d, "f_2"]\\ 
        \C^{3} \arrow[d, "f_3"] \\
        \C^{4} \arrow[d, "f_4"] \arrow[r, "g"] & \C^{2}\\ 
        \C^{3} \arrow[d, "f_4"] \\
        \C^{2} \arrow[d, "f_5"] \\

    \end{tikzcd}
\end{center}

and generically, after a choice of basis
\begin{center}
    \begin{tikzcd}
        \C \arrow[d, "Id\oplus 0"] \\
        \C\oplus \C \arrow[d, "Id \oplus 0"]\\ 
        \C\oplus \C\oplus \C \arrow[d, "Id \oplus 0"] \\
        \C\oplus \C\oplus \C\oplus \C  \arrow[d, "\pi_1 \oplus \pi_2 \oplus \pi_3"] \arrow[to = 4-4, "\pi_2 +\pi_3 + \pi_4 \oplus \pi_1 + \pi_4"] & & & \C\oplus \C\\ 
        \C\oplus\C\oplus \C \arrow[d, "\pi_1 \oplus \pi_2"] \\
        \C\oplus \C \\

    \end{tikzcd}
\end{center}

\subsubsection{$E_8$}
For the highest $E_8$ root we have:
\begin{center}
    \begin{tikzcd}
        \C^{2} \arrow[d, "f_1"]\\ 
        \C^{3} \arrow[d, "f_2"] \\
        \C^{4} \arrow[d, "f_3"] \\ 
        \C^{5} \arrow[d, "f_4"] \\
        \C^{6} \arrow[d, "g_1"] \arrow[r, "h"] & \C^{3}\\
        \C^{4} \arrow[d, "g_2"] \\
        \C^{2} &\\  
       
    \end{tikzcd}
\end{center}
and generically, a choice of basis gives:
\begin{center}
    \begin{tikzcd}
        \C\oplus\C \arrow[d, "Id\oplus 0"]\\ 
        \C\oplus\C\oplus\C \arrow[d, "Id\oplus 0"] \\
        \C\oplus\C\oplus\C\oplus\C \arrow[d, "Id\oplus 0"] \\ 
        \C\oplus\C\oplus\C\oplus\C\oplus\C \arrow[d, "Id\oplus 0"] \\
        \C\oplus\C\oplus\C\oplus\C\oplus\C\oplus\C \arrow[d, "\pi_1\oplus \pi_2 \oplus \pi_3 \oplus\pi_4"] \arrow[r, "h"] & \C\oplus\C\oplus\C\\
        \C\oplus\C\oplus\C\oplus\C \arrow[d, "\pi_1\oplus\pi_2 "] \\
        \C\oplus\C &\\  
       
    \end{tikzcd}
\end{center}
Where \[h = \begin{pmatrix}
    1 & 0 & 0 & 0 & 1 & 1\\
    0 & 1 & 0 & 1 & 1 & 0 \\
    0 & 0 & 1 & 0 & 1 & 0\\
\end{pmatrix}\]

\section{Quiver Representations for the Affine ADE Roots}\label{sec:affinequiver}
\subsection{$\tilde{A}_n$}
The imaginary roots of the $\tilde{A}_n$ algebras have associated quivers with the structure
\begin{center}
    \begin{tikzcd}
               & \C^n \arrow[r, "f_1"]& \hdots \arrow[r, "f_i"]& \C^n \arrow [dr, "f_{i+1}"]& \\
        \C^n \arrow[ur, "f_{n}"]  &      &        &      & \C^n\arrow[d, "f_{i+2}"]\\
        \vdots\arrow[u, "f_{n-2}"]  &      &        &      & \vdots\arrow[d, "f_{k-1}"]\\
        \C^n\arrow[u, "f_{j+1}"]   &      &        &      & \C^n\arrow[dl, "f_k"]\\
               & \C^n\arrow[ul, "f_{j}"] & \hdots\arrow[l, "f_{j-1}"] & \C^n\arrow[l, "f_{k+1}"] & \\
    \end{tikzcd}
\end{center}
We assume generic ranks for all maps and their compositions (namely that each map $f_i$ should be invertible, we may then select a basis such that the quiver representation is isomorphic to:
\begin{center}
    \begin{tikzcd}
               & \C^n \arrow[r, "Id"]& \hdots \arrow[r, "Id"]& \C^n \arrow [dr, "Id"]& \\
        \C^n \arrow[ur, "J"]  &      &        &      & \C^n\arrow[d, "Id"]\\
        \vdots\arrow[u, "Id"]  &      &        &      & \vdots\arrow[d, "Id"]\\
        \C^n\arrow[u, "Id"]   &      &        &      & \C^n\arrow[dl, "Id"]\\
               & \C^n\arrow[ul, "Id"] & \hdots\arrow[l, "Id"] & \C^n\arrow[l, "Id"] & \\
    \end{tikzcd}
\end{center}
Where $J$ is any matrix in Jordan normal form. Generically this matrix is diagonal, and if we want the representation to be indecomposible $J$ is required to be a single Jordan block. This Jordan block structure is generically independent of orientation, as generically all of the maps will be invertible and thus one can work with the inverse rather than the map itself, making orientation arbitrary. If $J$ is a single Jordan block, it is isomorphic to some $J_\lambda$, and thus in the visual calculus for quiver representations we would find something which looks like:

\begin{figure}[h!]\label{fig:A_n affine quiver}
    \centering
    \begin{tikzpicture}[auto, node distance=2cm, transform shape, bend angle = 90]
                \node (dots1) [] {...};
                \node[state] (1) [right of = dots1] {1};
                \node[state] (2) [right of = 1] {2};
                \node (dots2) [right of = 2] {...};
                \node[state] (n) [right of = dots2] {n};
                \node[state] (1two) [right of = n] {1};
                \node (dots3) [right of = 1two] {...};

                \draw [arrows = {-Stealth[scale=2]}]  (dots1) -- (1) node[midway, above ]{1};
                \draw [arrows = {-Stealth[scale=2]}]  (1) -- (2) node[midway, above ]{1};
                \draw [arrows = {-Stealth[scale=2]}]  (2) -- (dots2) node[midway, above ]{1};
                \draw [arrows = {-Stealth[scale=2]}]  (dots2) -- (n) node[midway, above ]{1};
                \draw [arrows = {-Stealth[scale=2]}]  (n) -- (1two) node[midway, above ]{1};
                \draw [arrows = {-Stealth[scale=2]}]  (1two) -- (dots3) node[midway, above ]{1};
                \draw [arrows={-Stealth[scale=2]}, bend right]
  (n) to node[midway, above] {$\lambda$} (1);

        \end{tikzpicture}
    \caption{The indecomposible imaginary $\tilde{A}_n$ Representations}
\end{figure}
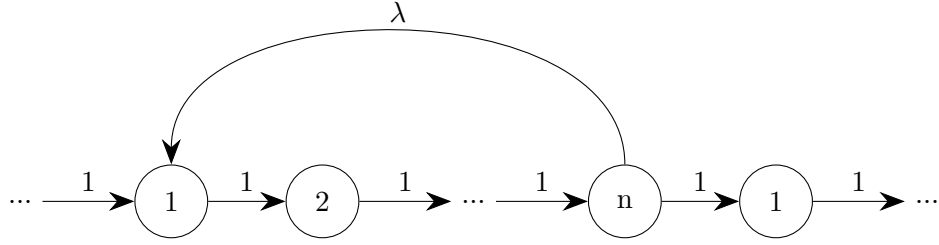
The only loops in this graph are caused by nonzero values of the parameter $\lambda$, and thus the representation becomes tree-like if this is parameter is taken to be zero. In general, the representative diagrams of positive roots of $\tilde{A_n}$ correspond to tree like diagrams of quivers. This process also shows a birational equivalence (for definitions the standard reference is probably Hartshorne \cite{hartshorne}, a much more on topic and detailed account of this point of view as applies to Nakajima's quiver varieties is \cite{bellamy2025birationalgeometryquivervarieties}) between the spaces of representations, with imaginary dimension vector with uniform value $k$, for arbitrary $\tilde{A}_n$, and the Jordan quiver which classifies maps $V\rightarrow V$ up to a choice of basis. This will be a recurring theme as we discuss more affine diagrams. 

\subsection{$\tilde{D}_n$ Quivers}
The imaginary roots of the $\tilde{D}_n$ algebras have the structure
\begin{center}
    \begin{tikzcd}
        \C^n \arrow[dr, "a"] &        &    &         & \C^n\\
             &\C^{2n} \arrow[r, "f_1"]&... \arrow[r, "f_{n-4}"] & \C^{2n}\arrow[ur, "b"]\arrow[dr, "d"] & \\
        \C^n\arrow[ur, "c"] &        &    &         & \C^n\\
    \end{tikzcd}
    
\end{center}
We assume generic ranks of all maps and compositions, we can choose a basis so that the resulting quiver representation is isomorphic to

\begin{center}
    \begin{tikzcd}
        \C^n \arrow[dr, "Id \oplus 0"] &        &    &         & \C^n\\
             &\C^{n}\oplus\C^n \arrow[r, "Id"]&... \arrow[r, "Id"] & \C^{n}\oplus \C^n\arrow[ur, "\pi_1 + \pi_2"]\arrow[dr, "\pi_1+J\circ\pi_2"] & \\
        \C^n\arrow[ur, "0\oplus Id"] &        &    &         & \C^n\\
    \end{tikzcd}
    
\end{center}
Again, if $J$ is a single Jordan block it will be an indecomposible representation, but the generic representation will have $J$ diagonal. If we have $J = J_\lambda$, the space of representations will look like:

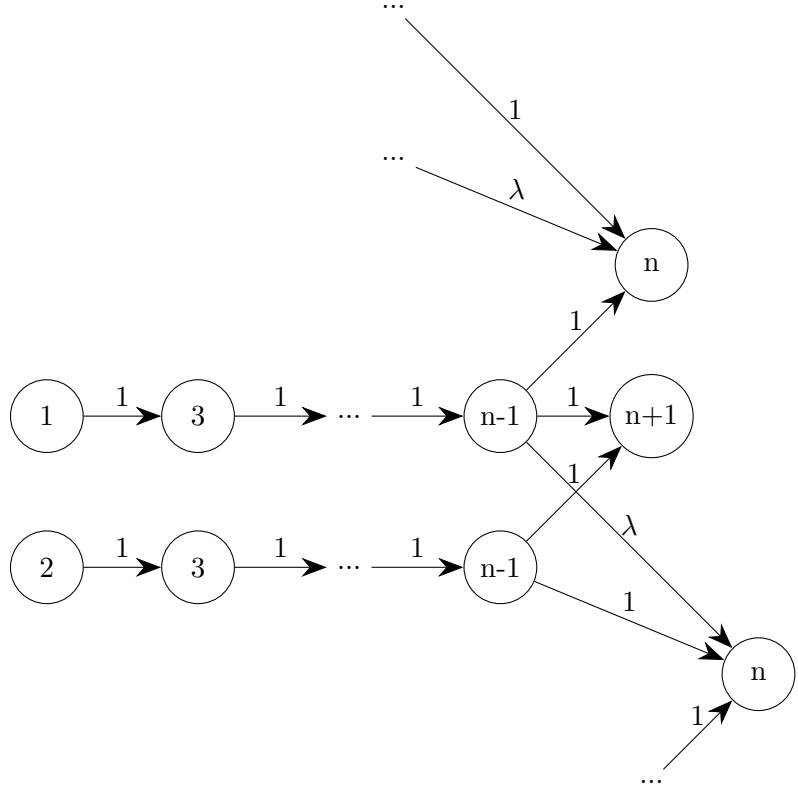
\begin{figure}[h!]\label{fig:D_n affine quiver}
    \centering
    \begin{tikzpicture}[auto, node distance=2cm, transform shape, bend angle = 90]
                \node[state] (1) [] {1};
                \node[state] (2) [below of = 1] {2};
                \node[state] (3one) [right of = 1] {3};
                \node[state] (3two) [right of = 2] {3};
                \node (dots1) [right of = 3one] {...};
                \node (dots2) [right of = 3two] {...};
                \node[state] (n-11) [right of = dots1] {n-1};
                \node[state] (n-12) [right of = dots2] {n-1};
                \node[state] (n+1) [right of = n-11] {n+1};
                \node[state] (n) [below right of = n-12,xshift = 2cm] {n};

                \draw [arrows = {-Stealth[scale=2]}]  (1) -- (3one) node[midway, above ]{1};
                \draw [arrows = {-Stealth[scale=2]}]  (2) -- (3two) node[midway, above ]{1};
                \draw [arrows = {-Stealth[scale=2]}]  (3one) -- (dots1) node[midway, above ]{1};
                \draw [arrows = {-Stealth[scale=2]}]  (3two) -- (dots2) node[midway, above ]{1};
                \draw [arrows = {-Stealth[scale=2]}]  (dots1) -- (n-11) node[midway, above ]{1};
                \draw [arrows = {-Stealth[scale=2]}]  (dots2) -- (n-12) node[midway, above ]{1};
                \draw [arrows = {-Stealth[scale=2]}]  (n-12) -- (n+1) node[midway, above ]{1};
                \draw [arrows = {-Stealth[scale=2]}]  (n-11) -- (n+1) node[midway, above ]{1};

                \node (n+1prev) [below left of = n] {...};
                \node[state] (nnext) [above of = n+1] {n};
                \node (n-1next) [above left of = nnext, xshift = -2cm] {...};
                \node (n-1next2) [above of = n-1next] {...};
                \draw [arrows = {-Stealth[scale=2]}]  (n-12) -- (n) node[midway, above ]{1};
                \draw [arrows = {-Stealth[scale=2]}]  (n-11) -- (n) node[midway, above ]{$\lambda$};
                \draw [arrows = {-Stealth[scale=2]}]  (n+1prev) -- (n) node[midway, above ]{1};
                \draw [arrows = {-Stealth[scale=2]}]  (n-11) -- (nnext) node[midway, above ]{1};
                \draw [arrows = {-Stealth[scale=2]}]  (n-1next) -- (nnext) node[midway, above ]{$\lambda$};
                \draw [arrows = {-Stealth[scale=2]}]  (n-1next2) -- (nnext) node[midway, above ]{1};

        \end{tikzpicture}
    \caption{The indecomposible imaginary $\tilde{D}_n$ Representations}
\end{figure}
This particular choice of parameterization does not clearly correspond to the iterable pattern we saw in \ref{fig:twospecial}, however when the parameter $\lambda$ is set to zero it may be unfolded and put in order as a related pattern starting with the vertex $n$ instead of $3$. This presentation was chosen for the sake of managing space on the page. Because the whole representation (generically) is parameterized by $J$, the space of such representations is \textit{also} birationally equivalent to the Jordan quiver.
\subsection{$\tilde{E}_n$ Quivers}
\paragraph{Quivers for the $\tilde{E}_6$ Imaginary Roots}
We have a quiver of the form:\\
\begin{center}
\begin{tikzcd}
   \C^n \arrow[r, leftarrow, "h_2"] & \C^{2n} \arrow[r, leftarrow, "h_1"] & \C^{3n} \arrow[r, "g_1"] \arrow[d, "f_1"] & \C^{2n} \arrow[r, "g_2"] &\C^n\\
   &  & \C^{2n} \arrow[d, "f_2"] & & \\
   &  & \C^n
\end{tikzcd}
\end{center}
Which, after making generic choices for ranks of all maps, a basis can be chosen to put it in the form:\\
\begin{center}
\begin{tikzcd}
   \C^n \arrow[r, leftarrow, "\pi_1 + \pi_2"] & \C^{n}\oplus \C^n \arrow[r, leftarrow, "\pi_2\oplus \pi_3"] & \C^n\oplus\C^n\oplus\C^n \arrow[r, "\pi_1\oplus \pi_2"] \arrow[d, "\pi_1 \oplus \pi_3"] & \C^{n}\oplus \C^n \arrow[r, "\pi_1 + \pi_2"] &\C^n\\
   &  & \C^{n}\oplus \C^n \arrow[d, "\pi_1 + J\circ\pi_2"] & & \\
   &  & \C^n
\end{tikzcd}
\end{center}
Where $\pi_i$ is always projection on to the $i$-th part, and $J$ is any matrix in Jordan normal form. Generically this matrix will be diagonal, and so the representation will be decomposible. The indecomposible case occurs if and only if $J$ comprises a single Jordan block $J_\lambda$. The graphical calculus notation for this is 

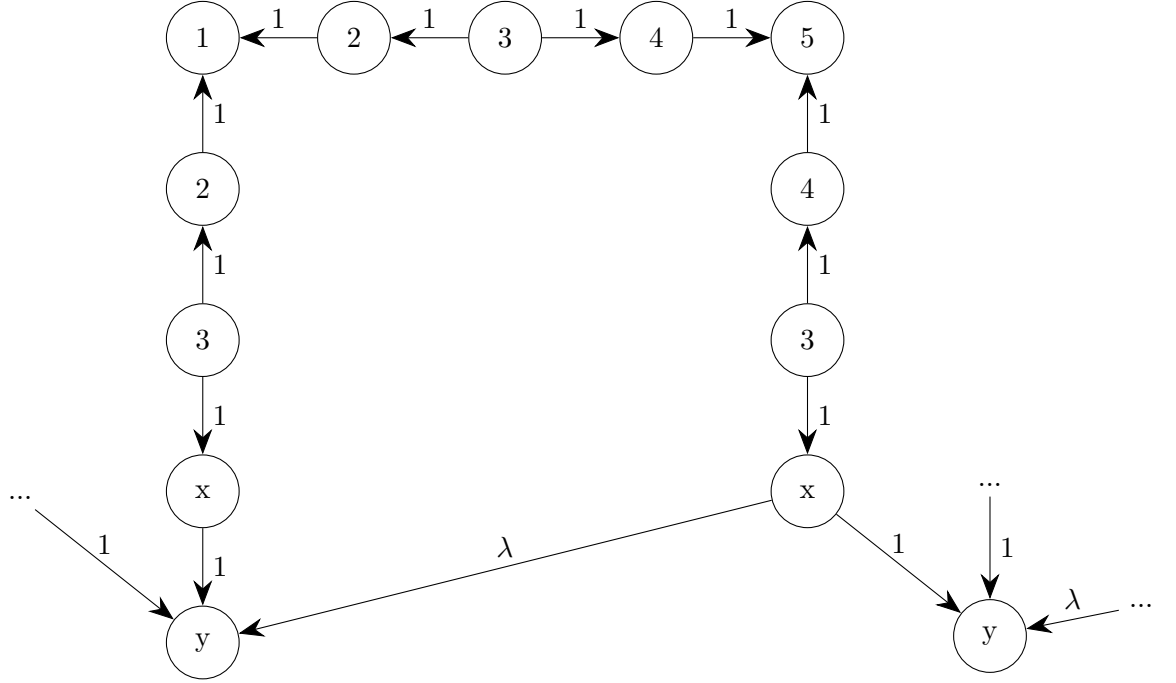
\begin{figure}[h!]\label{fig:E_6 affine}
    \centering
    \begin{tikzpicture}[auto, node distance=2cm, transform shape, bend angle = 90]
                \node[state] (1) [] {1};
                \node[state] (2one) [below of = 1] {2};
                \node[state] (2two) [right of = 1] {2};
                \node[state] (3top) [right of = 2two] {3};
                \node[state] (4one) [right of = 3top] {4};
                \node[state] (5) [right of = 4one] {5};
                \node[state] (4two) [below of = 5] {4};

                \node[state] (3two) [below of = 2one] {3};
                \node[state] (xone) [below of = 3two] {x};
                \node[state] (yone) [below of = xone] {y};
                \node[state] (3three) [below  of = 4two] {3};
                \node[state] (xtwo) [below of = 3three] {x};
                
                \node[state] (ynext) [below right of = xtwo, xshift = 1cm, yshift = -5mm] {y};
                \node (xlast) [above left of = yone, xshift = -1cm, yshift = 5mm] {...};
                \node (xnext) [above of = ynext] {...};
                \node (xnext2) [right of = ynext, yshift = 4mm] {...};

                \draw [arrows = {-Stealth[scale=2]}]  (2one) -- (1) node[midway, right ]{1};
                \draw [arrows = {-Stealth[scale=2]}]  (2two) -- (1) node[midway, above ]{1};
                \draw [arrows = {-Stealth[scale=2]}]  (3top) -- (2two) node[midway, above ]{1};
                \draw [arrows = {-Stealth[scale=2]}]  (3top) -- (4one) node[midway, above ]{1};
                \draw [arrows = {-Stealth[scale=2]}]  (4one) -- (5) node[midway, above ]{1};
                \draw [arrows = {-Stealth[scale=2]}]  (3two) -- (2one) node[midway, right ]{1};
                \draw [arrows = {-Stealth[scale=2]}]  (3two) -- (xone) node[midway, right ]{1};
                \draw [arrows = {-Stealth[scale=2]}]  (xone) -- (yone) node[midway, right ]{1};
                \draw [arrows = {-Stealth[scale=2]}]  (xtwo) -- (yone) node[midway, above ]{$\lambda$};
                \draw [arrows = {-Stealth[scale=2]}]  (4two) -- (5) node[midway, right ]{1};
                \draw [arrows = {-Stealth[scale=2]}]  (3three) -- (4two) node[midway, right ]{1};
                \draw [arrows = {-Stealth[scale=2]}]  (3three) -- (xtwo) node[midway, right ]{1};
                \draw [arrows = {-Stealth[scale=2]}]  (xtwo) -- (ynext) node[midway, above ]{1};
                \draw [arrows = {-Stealth[scale=2]}]  (xlast) -- (yone) node[midway, above ]{1};
                \draw [arrows = {-Stealth[scale=2]}]  (xnext) -- (ynext) node[midway, right ]{1};
                \draw [arrows = {-Stealth[scale=2]}]  (xnext2) -- (ynext) node[midway, above ]{$\lambda$};

        \end{tikzpicture}
    \caption{The indecomposible imaginary $\tilde{E}_6$ Representations}
\end{figure}
\newpage
This bears very little resemblance to any form of the iterable pattern for $\tilde{E}_6$ affine that the author is aware of, but does have the same numbers of generators. We see yet again, however, as again $J$ parameterizes the space of (generic) representations, that the space of these representations is birational to that of the Jordan quiver.
\paragraph{Quivers for the $\tilde{E}_7$ Imaginary Roots}
For imaginary roots, we choose a quiver of the form:
\begin{center}
    \begin{tikzcd}
        \C^{n} \arrow[d, "f_1"] \\
        \C^{2n} \arrow[d, "f_2"]\\ 
        \C^{3n} \arrow[d, "f_3"] \\
        \C^{4n} \arrow[d, "f_4"] \arrow[r, "g"] & \C^{2n}\\ 
        \C^{3n} \arrow[d, "f_4"] \\
        \C^{2n} \arrow[d, "f_5"] \\
        \C^{n} &\\  
       
    \end{tikzcd}
\end{center}
Making generic choices for ranks of maps allows us to choose a basis putting the quiver in the form:
\begin{center}
    \begin{tikzcd}
        \C^{n} \arrow[d, "Id\oplus 0"] \\
        \C^{n}\oplus \C^n \arrow[d, "Id\oplus 0"]\\ 
        \C^{n}\oplus \C^n\oplus \C^n \arrow[d, "Id\oplus 0"] \\
        \C^{n}\oplus \C^n\oplus \C^n\oplus \C^n \arrow[d, "\pi_1\oplus \pi_2 \oplus \pi_3"] \arrow[r, "g"] & \C^{n}\oplus \C^n\\ 
        \C^{n}\oplus \C^n\oplus \C^n \arrow[d, "\pi_1\oplus\pi_2"] \\
        \C^{n}\oplus \C^n \arrow[d, "\pi_1"] \\
        \C^{n} &\\  
       
    \end{tikzcd}
\end{center}
And we give $g$ in block form as

\[
g = \begin{pmatrix}
    J & Id & 0 & Id\\
    Id & Id & Id & 0
\end{pmatrix}
\]
Where $J$ is in Jordan form. Again generically this means $J$ is diagonal and indecomposible means it is one Jordan block. If that Jordan block is $J\lambda$, we see:
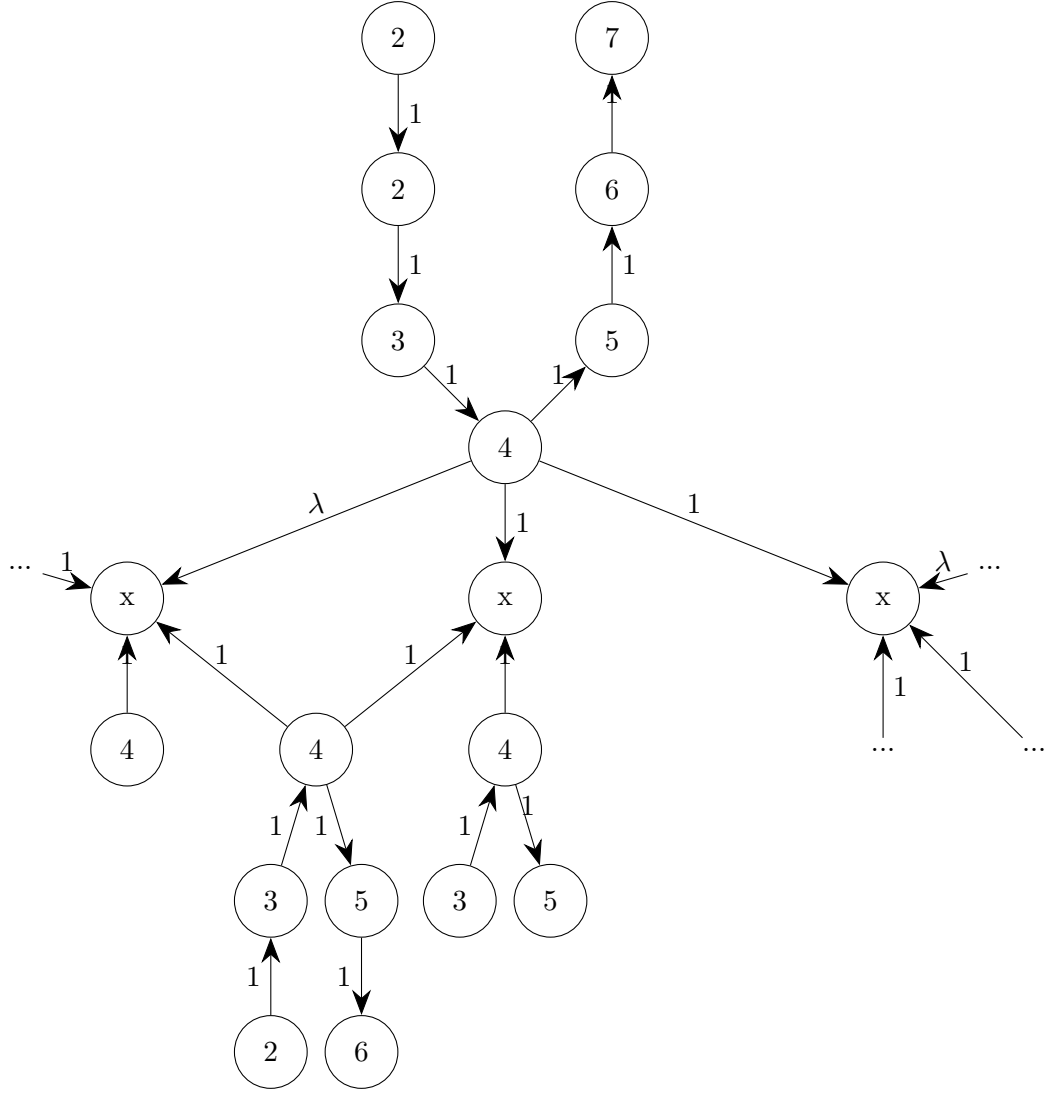
\begin{figure}[h!]\label{fig:E_7 affine quiver}
    \centering
    \begin{tikzpicture}[auto, node distance=2cm, transform shape, bend angle = 90]
                \node[state] (xone) [] {x};
                \node[state] (xtwo) [right of = xone, xshift = 3cm] {x};
                \node[state] (xnext) [right of = xtwo, xshift = 3cm] {x};

                \node[state](4one) [below of = xone] {4};
                \node[state](4three) [right of = 4one, xshift = 5mm] {4};
                \node[state] (4two) [right of = 4three, xshift = 5mm] {4};
                \node[state] (4four) [above of = xtwo] {4};
                \draw [arrows = {-Stealth[scale=2]}]  (4one) -- (xone) node[midway, above ]{1};
                \node[state] (5one) [below of = 4two, xshift = 6mm] {5};
                \node[state] (3one) [below of = 4two, xshift = -6mm] {3};
                \draw [arrows = {-Stealth[scale=2]}]  (3one) -- (4two) node[midway, left ]{1};
                \draw [arrows = {-Stealth[scale=2]}]  (4two) -- (5one) node[midway, above ]{1};
                \draw [arrows = {-Stealth[scale=2]}]  (4two) -- (xtwo) node[midway, above ]{1};
                \node[state] (5two) [below of = 4three, xshift = 6mm] {5};
                \node[state] (3two) [below of = 4three, xshift = -6mm] {3};
                \node[state] (2two) [below of = 3two] {2};
                \node[state] (6two) [below of = 5two] {6};
                \draw [arrows = {-Stealth[scale=2]}]  (2two) -- (3two) node[midway, left ]{1};
                \draw [arrows = {-Stealth[scale=2]}]  (3two) -- (4three) node[midway, left ]{1};
                \draw [arrows = {-Stealth[scale=2]}]  (4three) -- (5two) node[midway, left ]{1};
                \draw [arrows = {-Stealth[scale=2]}]  (5two) -- (6two) node[midway, left ]{1};
                \draw [arrows = {-Stealth[scale=2]}]  (4three) -- (xone) node[midway, above ]{1};
                \draw [arrows = {-Stealth[scale=2]}]  (4three) -- (xtwo) node[midway, above ]{1};
                \node[state] (5three) [above right of = 4four] {5};
                \node[state] (3three) [above left of = 4four] {3};
                \node[state] (2three) [above of = 3three] {2};
                \node[state] (6three) [above of = 5three] {6};
                \node[state] (1) [above of = 2three] {2};
                \node[state] (7) [above of = 6three] {7};
                \draw [arrows = {-Stealth[scale=2]}]  (1) -- (2three) node[midway, right ]{1};
                \draw [arrows = {-Stealth[scale=2]}]  (2three) -- (3three) node[midway, right ]{1};
                \draw [arrows = {-Stealth[scale=2]}]  (3three) -- (4four) node[midway, above ]{1};
                \draw [arrows = {-Stealth[scale=2]}]  (4four) -- (5three) node[midway, above ]{1};
                \draw [arrows = {-Stealth[scale=2]}]  (5three) -- (6three) node[midway, right ]{1};
                \draw [arrows = {-Stealth[scale=2]}]  (6three) -- (7) node[midway, above ]{1};
                \draw [arrows = {-Stealth[scale=2]}]  (4four) -- (xone) node[midway, above ]{$\lambda$};
                \draw [arrows = {-Stealth[scale=2]}]  (4four) -- (xtwo) node[midway, right ]{1};
                \draw [arrows = {-Stealth[scale=2]}]  (4four) -- (xnext) node[midway, above ]{1};
                \node (4last) [above left of = xone, yshift = -1cm] {...};
                \draw [arrows = {-Stealth[scale=2]}]  (4last) -- (xone) node[midway, above ]{1};
                \node (4onenext) [below of = xnext] {...};
                \draw [arrows = {-Stealth[scale=2]}]  (4onenext) -- (xnext) node[midway, right ]{1};
                \node (4twonext) [right of = 4onenext] {...};
                \draw [arrows = {-Stealth[scale=2]}]  (4twonext) -- (xnext) node[midway, above ]{1};
                \node (4fournext) [above right of = xnext, yshift = -1cm] {...};
                \draw [arrows = {-Stealth[scale=2]}]  (4fournext) -- (xnext) node[midway, above ]{$\lambda$};

        \end{tikzpicture}
    \caption{The indecomposible imaginary $\tilde{E}_7$ Representations}
\end{figure}
As we have seen with the $\tilde{E}_6$ quiver, we again have a birational equivalence with representations of the Jordan quiver. Furthermore, we once again have a structure which bears little resemblance (beyond multiplicities) to the iterable element in the associated Lie algebra.
\newpage
\paragraph{Quivers for the $\tilde{E}_8$ Imaginary Roots}
The quiver we choose is of the form
\begin{center}
    \begin{tikzcd}
        \C^{n} \arrow[d, "f_1"] \\
        \C^{2n} \arrow[d, "f_2"]\\ 
        \C^{3n} \arrow[d, "f_3"] \\
        \C^{4n} \arrow[d, "f_4"] \\ 
        \C^{5n} \arrow[d, "f_5"] \\
        \C^{6n} \arrow[d, "g_1"] \arrow[r, "h"] & \C^{3n}\\
        \C^{4n} \arrow[d, "g_2"] \\
        \C^{2n} &\\  
       
    \end{tikzcd}
\end{center}
Once more, generic choices for ranks and specific choices of basis allow us to put the quiver in the form:
\begin{center}
    \begin{tikzcd}
        \C^{n} \arrow[d, "Id\oplus 0"]\\
        \C^{n}\oplus \C^n \arrow[d, "Id\oplus 0"]\\ 
        \C^{n}\oplus \C^n\oplus \C^n \arrow[d, "Id\oplus 0"] \\
        \C^{n}\oplus \C^n\oplus \C^n\oplus \C^n \arrow[d, "Id\oplus 0"] \\ 
        \C^{n}\oplus \C^n\oplus \C^n\oplus \C^n\oplus \C^n \arrow[d, "Id\oplus 0"] \\
        \C^{n}\oplus \C^n\oplus \C^n\oplus \C^n\oplus \C^n\oplus \C^n \arrow[d, "\pi_1\oplus \pi_2 \oplus \pi_3  \oplus \pi_4"] \arrow[r, "h"] & \C^{n}\oplus \C^n\oplus \C^n\\
        \C^{n}\oplus \C^n\oplus \C^n\oplus \C^n \arrow[d, "\pi_1\oplus \pi_2"] \\
        \C^{n}\oplus \C^n  
       
    \end{tikzcd}
\end{center}
Where $h$ is given in block form by
\[h = 
\begin{pmatrix}
 Id & Id & Id & 0  & 0  & 0 \\
 Id & 0  & 0  & 0  & Id & Id  \\
 0  & 0  & Id & Id & J  & 0 \\
\end{pmatrix}
\]
Where $J$ again is any $n\times n$ matrix in Jordan form. Again, the representation is indecomposible iff $J$ is a single Jordan block $J_{\lambda}$, and again $J$ is generically diagonal. In our graphical calculus, we obtain

\begin{figure}[h!]\label{fig:E_6 affine quiver}
    \centering
    \begin{tikzpicture}[auto, node distance=2cm, scale = 0.5, transform shape, bend angle = 90]
               \node[state] (x1) [] {x};
               \node[state] (53) [above left of = x1] {5};
               \node[state] (54) [above right of = x1] {5};
               \node[state] (x2) [above right of = 54]  {x};
               \node[state] (55) [above left of = x2] {5};
               \node[state] (56) [right of = x2, xshift  = 2cm] {5};
               \node[state] (x3) [below right of = 56] {x};
               \node[state] (51) [above right of = x3] {5};
               \node[state] (52) [below right of = x3] {5};

               \draw [arrows = {-Stealth[]}]  (53) -- (x1) node[midway, above ]{1};
               \draw [arrows = {-Stealth[]}]  (54) -- (x1) node[midway, above ]{1};
               \draw [arrows = {-Stealth[]}]  (54) -- (x2) node[midway, above ]{1};
               \draw [arrows = {-Stealth[]}]  (55) -- (x2) node[midway, above ]{1};
               \draw [arrows = {-Stealth[]}]  (56) -- (x2) node[midway, above ]{1};
               \draw [arrows = {-Stealth[]}]  (56) -- (x3) node[midway, above ]{1};
               \draw [arrows = {-Stealth[]}]  (51) -- (x3) node[midway, above ]{1};
               \draw [arrows = {-Stealth[]}]  (52) -- (x3) node[midway, above ]{1};
               \draw [arrows = {-Stealth[]}]  (52) -- (x1) node[midway, above ]{$\lambda$};

               \node (5last) [left of = x1, xshift = -4cm] {...};
               \node (xnext) [right of = 52, xshift = 1cm] {...};
                \draw [arrows = {-Stealth[]}]  (52) -- (xnext) node[midway, above ]{1};
                 \draw [arrows = {-Stealth[]}]  (5last) -- (x1) node[midway, above ]{1};

               \node[state] (46) [above left of = 56] {4};
               \node[state] (36) [above left of = 46] {3};
               \node[state] (26) [above of = 36] {2};
               \node[state] (16) [above of = 26] {1};
               \node[state] (06) [above of = 16] {0};
               \node[state] (66) [above right of = 56] {6};
               \node[state] (76) [above right of = 66] {7};

               \draw [arrows = {-Stealth[]}]  (06) -- (16) node[midway, right ]{1};
               \draw [arrows = {-Stealth[]}]  (16) -- (26) node[midway, right ]{1};
               \draw [arrows = {-Stealth[]}]  (26) -- (36) node[midway, right ]{1};
               \draw [arrows = {-Stealth[]}]  (36) -- (46) node[midway, above ]{1};
               \draw [arrows = {-Stealth[]}]  (46) -- (56) node[midway, above ]{1};
               \draw [arrows = {-Stealth[]}]  (56) -- (66) node[midway, above ]{1};
               \draw [arrows = {-Stealth[]}]  (66) -- (76) node[midway, above ]{1};

               \node[state] (45) [above left of = 55] {4};
               \node[state] (35) [above left of = 45] {3};
               \node[state] (25) [above left of = 35] {2};
               \node[state] (15) [above left of = 25] {1};
               \node[state] (65) [above of = 55] {6};
               \node[state] (75) [above of = 65] {7};

               \draw [arrows = {-Stealth[]}]  (15) -- (25) node[midway, above ]{1};
               \draw [arrows = {-Stealth[]}]  (25) -- (35) node[midway, above ]{1};
               \draw [arrows = {-Stealth[]}]  (35) -- (45) node[midway, above ]{1};
               \draw [arrows = {-Stealth[]}]  (45) -- (55) node[midway, above ]{1};
               \draw [arrows = {-Stealth[]}]  (55) -- (65) node[midway, right ]{1};
               \draw [arrows = {-Stealth[]}]  (65) -- (75) node[midway, right ]{1};

               \node[state] (44) [above left of = 54] {4};
               \node[state] (34) [above left of = 44] {3};
               \node[state] (24) [above left of = 34] {2};
               \node[state] (64) [right of = 54] {6};
               
               \draw [arrows = {-Stealth[]}]  (24) -- (34) node[midway, above ]{1};
               \draw [arrows = {-Stealth[]}]  (34) -- (44) node[midway, above ]{1};
               \draw [arrows = {-Stealth[]}]  (44) -- (54) node[midway, above ]{1};
               \draw [arrows = {-Stealth[]}]  (54) -- (64) node[midway, above ]{1};

               \node[state] (43) [left of = 53] {4};
               \node[state] (33) [above left of = 43] {3};
               \node[state] (63) [above left of = 53] {6};
               
               \draw [arrows = {-Stealth[]}]  (33) -- (43) node[midway, above ]{1};
               \draw [arrows = {-Stealth[]}]  (43) -- (53) node[midway, above ]{1};
               \draw [arrows = {-Stealth[]}]  (53) -- (63) node[midway, above ]{1};

               \node[state] (42) [above right of = 52] {4};
               \draw [arrows = {-Stealth[]}]  (42) -- (52) node[midway, above ]{1};

        \end{tikzpicture}
    \caption{The indecomposible imaginary $\tilde{E}_8$ Representations}
\end{figure}
\newpage
Of course, we again immediately obtain that the space of such representations is birationally equivalent to those of the Jordan quiver. In this case the resultant iterable pattern also does not appear to be related to that for the associated algebra. With that complete, the details of this section may be synthesized into the general proposition:

\begin{prop}\label{thm:affineloop}
    The Affine Dynkin diagrams of type $ADE$ have spaces of representations for the imaginary roots $nr$ (where $r$ is the smallest positive imaginary root), which are birationally equivalent to the space of representations of the Jordan quiver when it is assigned a vector space of rank $n$.
\end{prop}

Each imaginary root has a minimal imaginary root which divides it. For these, the actual indecomposible representation is described by a diagram which has a single loop, and a parameter which describes the coefficient of going around that loop. In general, any of the edges on the loop may be set to contain that parameter, and the rest can be set to one. One obtains tree like representations by setting that parameter to zero and thus breaking the loop. Tree like representations obtained in that way are called sub-generic \label{def:subgeneric}, as they are no longer technically generic by virtue of rank. Starting from a generic representation, a new representation may be obtained for any integer $n$ by isolating an edge, forming $n$ copies of the tree and adding edges between the $n$-th copy of the source of the isolated edge and the $n+1$-th copy of the target. This representation is indecomposable and is the generic indecomposable for the imaginary root given by $n$ times the minimal one. The number of tree-like sub-generic indecomposables obtained from them grows polynomially in $n$.
$$ dim(Im (g))\cap \mathrm{span}\{e_1,...,e_r\}) = \max (r+n - k,0)$$


\section{Binary Coding in the Quiver Setting}\label{sec:binaryquiver}
There are two ways in which the binary coding argument replicates itself in the setting of quiver representations. The first is by counting dimensions of representation spaces, and the second by constructing families of tree-like representations. The proposition \ref{prop:affinesareasplit} is true when restricted to the simply laced diagrams, so we will restrict our attention to the least complicated simply laced diagrams extending those. Our attention will therefore be on those which extend affine $ADE$ diagrams by a single vertex connected by a single edge. For the sake of convenience, we will assume that this edge is oriented out from the new vertex. More classical candidate notions might be hyperbolicity (the diagram containing no proper induced subgraphs which are not finite or affine type) or Minkowskian signature of the (diagonalized) Cartan matrix, but these are less informative areas of extension as our tools are based on the combinatorial order $\leq_c$, which reduces to the subgraph relation in the case of simply laced diagrams. There is some overlap, in particular the algebras $\tilde{A}_n^\wedge$, $\tilde{D}_n^\wedge$, and $\tilde{E}_n^\wedge$ obtained by adjoining that additional vertex to the affine vertex are hyperbolic for low values of $n$. A complete list of hyperbolic Kac-Moody algebras and their diagrams may be found in Carbone et al. \cite{Carbone_2010}. Both the dimension counting and tree-like representation constructions rely on the birational equivalences between the representation space and the representation space of the Jordan quiver. In particular, we will work with the case that the ranks of all vector spaces in the original quiver are those corresponding to the imaginary root, with the new node denoted by $x$, the space of representations under consideration will thus be equivalent to the representation space of a quiver with two vertices $x$ and $J$, a set of $n$ arrows from $x$ to $J$, and an arrow looped from $J$ to itself. It may look like:
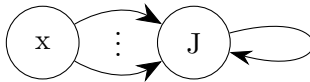
\begin{figure}[h]\label{fig:Jordan-n-quiver}
    \centering
    \begin{tikzpicture}[auto, node distance=2cm, transform shape]
            \node[state] (1) {x};
                \node[state] (2) [right of = 1] {J};

                \draw [arrows = {-Stealth[scale=2]}]  (1)  to[bend left] (2);
                \draw [arrows = {-Stealth[scale=2]}]  (1) to[bend right] node[yshift = 1mm] {$\vdots$} (2);
               \draw (2) edge[
    -{Stealth[length=3mm,width=2mm]},
    loop right,
    min distance=15mm,
    looseness=8
  ] (2);

        \end{tikzpicture}
    \caption{A modified Jordan quiver the addition of a single vertex to an affine quiver}
\end{figure}\\
Positive imaginary roots for the affine algebras are all positive integer multiples of fixed smallest imaginary positive roots. If $x$ is connected to a vertex in the affine diagram which is associated with a generator $j$, the number of arrows $n$ from $x$ to $J$ in the modified Jordan quiver will be equal to the multiplicity of the generator $j$ in the smallest imaginary positive root in that direction. This will also be the number of Jordan blocks into which the associated vector space may be divvied, while the size of those Jordan blocks will be the integer scaling from this smaller imaginary positive root to the larger one in question. The representation space will therefore be determined by three numbers, $n$ the number of times the generator appears in the smallest imaginary positive root, $k$ the size of the Jordan block, which is also the aforementioned scaling, and $r$ the rank of the vector space $x$. We may assume that $r\leq nk$, as otherwise the map from the vector space associated with $x$ to that associated with $j$ would have a nontrivial kernel, which could be amputated from $x$ in a direct sum with no loss of information. Thus, the representation space is equivalent to the set of maps from a space of dimension $r$ to a space of dimension $nk$. More interesting (and the reason why we write it as a modified Jordan quiver) is that once we have fixed that all the vector spaces from the original affine quiver are in the form discussed in the previous discussion, the remaining freedom we have in choosing a basis for the new quiver is given by acting with any invertible $r\times r$ matrix on the right and a matrix which acts on subspaces corresponding to Jordan blocks as $k\times k$ upper triangular Toeplitz matrices on the left. By this observation, it is manifest that any such representation is described by a set of $n$ distinct $k\times r$ matrices, up to the action of $r\times r$ invertibles on the right and $k\times k$ upper triangular Toeplitz matrices on the left. Assuming that the $r\times k$ matrices chosen collectively form a generic $nk\times r$ matrix before that equivalence is made, the first $r$-columns of it may be made into the $r\times r$ identity matrix by fixing the $r\times r$ matrix freedom. A choice of action by upper triangular Toeplitz matrices can be made to eliminate at most $k-1$ remaining degrees of freedom via row operations. Thus the dimension of the quotient space is given by $\max(rkn -r^2 -k-1,0)$, which is linear in each variable (keeping in mind the upper bound on $r$), which is commensurate with an exponential number of options (in this analogy the ``base" of the exponentiation would be $\C$).

If we examine the tree like representations, the representation restricted to the affine part of the quiver must also be tree-like (although it need not be connected, but this imposes additional conditions on how the $x$-vertices are connected in the graphical notation). In the graphical calculus, tree-likeness requires that each $x$-vertex not be attached twice to linked Jordan block subspaces. If we assume that the affine representation \textit{is} indecomposible, then the number of possible distinct trees is the number of $j$-vertices ($nk$) choose the number of $x$-vertices $r$. Meaning, the number of total trees with at most $nk$ $x$-vertices is $2^{nk}$, which gives the aforementioned exponential growth again, and recovers the graphical representations from the Kac-Moody algebra case if the underlying affine algebra is $A$ or $D$ type. If it is one of the $E$-types, then it becomes a question of whether each appropriate vertex in each iterate of the repeatable \textit{quiver representation} pattern is marked by adjoining an $x$ or not.

In the case where the Affine quiver representations are not indecomposible themselves, the number of tree-like representations is equivalent to the number of solutions to the following problem:
\begin{prob}
For each connected Jordan block subspace, create a vertex labelled $a$. For each for each $j$ vertex corresponding to a basis element of a Jordan block subspace, number those vertices $j_i$ for $i$ ranging between one and the total number, and connect it to the corresponding $a$ vertex. Finally, for each integer between $1$ and $r$, create a vertex labeled $c$. How many distinct trees can be made by connecting $j$ and $c$ vertices?
\end{prob}

While a closed form solution to this problem is not known to the author, the exponential estimate given by assuming the underlying affine representations are indecomposible does still show an exponential growth rate. This allows us to state the following theorem:
\begin{thm}
   Let $G$ be a simply laced graph. A quiver $Q$ constructed by orienting it will have a number of sub-generic tree-like representations which is
    
    \begin{enumerate}
       \item Finite, if the Kac-Moody algebra associated to the graph is simple (in this case a $1-1$ correspondence holds with positive roots)
       \item Infinite but grows as a polynomial function in the dimension vector, in the case that the Kac-Moody algebra associated to the graph is affine
      \item Infinite and grows exponentially, if the Kac-Moody algebra associated to the graph is of indefinite type.
   \end{enumerate}
\end{thm}
\newpage
\section{Conclusion}\label{sec:conc}
In this paper we have generated diagrams following two procedures:
\begin{itemize}
    \item For roots of Kac-Moody algebras: we add successive applications of simple roots until the process terminates or repeats.
    \item For quiver representations: we assume generic sizes of all definable subspaces and work from there to pick a basis which is maximally rigid, ideally then finding the generic indecomposable representation and from there relaxing the constraint to pick out the tree like representations.
\end{itemize}
The data here illustrate the Gabriel-Kac theorem, and indeed the root space multiplicity part of Kac's conjectures, by putting them in a common combinatorial language.

Indeed, in the case of quiver representations at indivisible imaginary roots, which is the case for which Crawley-Bovey's argument using deformed preprojective algebras \cite{Crawley-Boevey_2001} works, the combinatorial approach we illustrate gives an easy illustration for how to obtain a representation for multiples of that root:
\begin{enumerate}
    \item Start by finding the generic indecomposable for an indivisible imaginary root. It will have at least one loop in its minimal loop representation
    \item For each loop, concentrate all parameters on a single edge by scaling basis elements
    \item Pick one edge which still has a parameter.
    \item Replace all vertices with vector spaces $\C^n$
    \item Replace all other edges with matrices $Id$
    \item Replace the selected edge with a Jordan block $J_\lambda$
\end{enumerate}
the resultant representation is indecomposable and has the desired dimensions.

If a representation has $n$-loops with parameters which may each be freely chosen, which we might think of as being in analogy with independent loops of $\pi_1$ of a graph, or independent features in the $1$-homology groups, that representation is naturally a member of an $n$-dimensional family of representations. In particular, if those representations are all absolutely indecomposable, then they contribute a term of $(q-1)^n$ to the count of absolutely indecomposable representations over $\mathbb{F}_q$, albeit with some redundancy between such families, and where setting any of those parameters to zero is taken to relegate it to a smaller family. In particular, in this interpretation it is entirely clear that the tree-like representations are those which contribute to the constant term, and thus correspond to the root multiplicity. It is clear from the accounting done in the previous six sections that the tree-like representations obey the same asymptotic growth rate trichotomy as the root multiplicities, so may be an approximate algebraic model for them.

\subsection{Future Goals}
The combinatorial structures presented here illustrate a part of the theory which, while deep, is relatively primitive compared to more modern developments, including the constructions of the representation theory of Kac-Moody algebras using Nakajima's quiver varieties \cite{NakajimaKM}\cite{Hausel2010}, and constructions of enveloping algebras using Hall algebras of quiver representation categories \cite{Bridgeland2013}. To the author's knowledge, neither these more advanced results nor Kac's conjectures have been approached combinatorially using just the Dynkin diagrams before, and it would be of interest to see what comes of that line of inquiry. In both the Kac-Moody algebra case and the quiver representation case, we have managed to produce a binary coding argument, but have not produced a computationally tractable combinatorial model for either. The long-term goal of this project is to produce such a model, and pursuing these other lines of inquiry may help in doing so.

\subsection{Acknowledgments}
The author would like to thank John Baez for advice on how to present the material, his advisor Adrian Ocneanu for guidence, and the graduate student body of Penn State University and the organizers of the 2025 Symmetry, Discrete Geometry, and Mathematical Physics conference in Scalea, Italy for the opportunity to present these ideas to an audience.

\newpage
\addcontentsline{toc}{section}{References}
\printbibliography
\newpage
\appendix
\section*{Appendices}
\addcontentsline{toc}{section}{Appendices}
\section{Dynkin Diagrams, Cartan Matrices, and the Serre Relations}\label{sec:farbackground}
 This section is dedicated to the review of results which are used in the development of our method, but which would distract from the central narrative.

 \begin{dfn}\label{def:cart}
     A \textbf{genralized Cartan matrix} is an integral square matrix $A = (A_{ij})$ such that 
    \begin{itemize}
        \item $A_{ii} = 2$
        \item $A_{ij} \leq 0$ if $i\neq j$
        \item $A_{ij} = 0 \longleftrightarrow A_{ji} = 0$
    \end{itemize}
 \end{dfn}

Given such a matrix $A_{ij}$, we may define a Lie algebra with generators $e_i$, $f_i$, and $h_i$ for each row (we must also include a number of additional $k_n$ to describe the rank-deficiency of $A$). We start with the free Lie algebra generated by these, which does not include any of the information of the matrix itself besides its size. The remainder of the information is used to create relations:
\begin{dfn}\label{def:serre}
    The \textbf{Serre relations} are the following
     \begin{enumerate}
        \item $ad_{e_i}^{1-A_{ij}} e_j = 0$
        \item $ad_{f_i}^{1-A_{ij}}f_j = 0$
        \item $ad_{h_i}e_j = A_{ij}e_j$
        \item $ad_{h_i} f_j = - A_{ij}f_j$
        \item $[e_i, f_j] = \delta_{ij} h_i$
        \item $[h_i,h_j]=[h_i,k_n] = [k_n,k_m]=0$
    \end{enumerate}
\end{dfn}
$e$ and $f$ generators are also assumed to be eigenvectors for the $k_n$, in such a way as to ensure the independence of the eigenvalues as elements of the dual space of $ \mathcal{h}=\mathrm{span}\{h_1,h_2...,h_n, k_1... k_m\}$. Common eigenvalues of $h$ in the adjoint representation are called \textbf{roots}.
\begin{dfn}
    An algebra defined using these generators and relations is called the \textbf{Kac-Moody} algebra of $A$
\end{dfn}

The familiar Lie algebras arise in this way, as do the more exotic Affine Lie algebras used in conformal field theory. In fact, the familiar Lie algebras correspond to some sparse matrices with small entries. The diagonal contains no information, so we may rid ourselves of it by taking $2-A$, which will be a sparse, integral matrix with zero diagonal. Such a thing looks very much like the adjacency matrix of a graph (which here means a directed multigraph with no self-loops). If we take that point of view seriously, we get the notion of Dynkin diagrams.

\begin{dfn}
    A \textbf{Dynkin diagram} is a graph generated from a Cartan matrix in the following way:
     \begin{enumerate}
        \item Give a node for each row
        \item If $A_{ij} = A_{ji} = 0$, no edge between them
        \item If $A_{ij} = A_{ji} = -n$, $n$ undirected edges between them,
        \item If $A_{ij} = -n < A_{ji} = -1$, $n$ edges between them pointing from $j$ to $i$
        \item If $A_{ij}, A_{ji} < -1$, put an edge labled with $(A_{ij}, A_{ji})$ when read from $i$ to $j$

    \end{enumerate}
\end{dfn}
    
We differ from some authors, for instance Saclioglu \cite{hyperbolic} in points $3$ and $5$. These points are necessary, however, for visual clarity when more complicated algebras are under discussion. 

 We are going to use them to explicitly construct the elements of Kac-Moody algebras.

\begin{dfn}
    A generalized Cartan matrix $A$ is \textbf{symmetrizable} if there exists a diagonal matrix $D$ such that $DA$ is symmetric.
\end{dfn}

Symmetrizability is important because the roots generate a free abelian subgroup of $\C^n$ called the \textbf{root lattice}, and symmetrizability allows us to put a natural symmetric bilinear product on that group, turning it into a lattice in the sense of number theory. It allows us to work with the root lattice as a set of coordinates and reason about them geometrically, to wit Kac-Moody algebras are generated by three types of things: 
    \begin{itemize}
        \item $h_i$'s and $k_m$'s, which have joint eigenvalues that act like coordinate functions
        \item $e_i$'s, which act like raising operators for these values
        \item $f_i$'s, which act like lowering operators
    \end{itemize}
If we want to work with these generalized coordinates in any meaningful way, we want them to behave like coordinates, and therefore want roots to have defined lengths (or at least sensible relative lengths). Symmetrizability lets us do this. Many properties of Kac-Moody algebras, such as the existence of a Bernstein-Gelfand-Gelfand resolution, are true independent of symmetrizability, but require much harder proofs \cite{Kumar1990}.

\begin{dfn}
A diagram is called \textbf{simply laced} if all of its edges are single edges, or equivalently all non-diagonal entries in the associated Cartan matrix are $0$ or $-1$.
\end{dfn}

There is a well known trick due to Kac which allows for the construction of non simply laced graphs from simply laced ones with built in symmetries. The summary due to John Stembridge \cite{Stembridge_2008} is very helpful to understand this trick in more detail. On the level of algebras, this trick amounts to the realization that given generators $x$ and mutually commuting $e_1,...,e_n$, such that 
\begin{itemize}
    \item $Ad_{e_i}^{k}(x) =0$
    \item $Ad_{x}^{j}(e_i) = 0$    
\end{itemize}
and that the Lie algebra is freely generated besides these constraints, for the sum $s = e_1 + e_2 +... +e_n$, we have
\begin{itemize}
    \item $Ad_{s}^{nk}(x) = 0$
    \item $Ad_x^j(s)=0$
\end{itemize}
and the Lie algebra jointly generated between those two elements is otherwise free. This allows us to consolidate edges in the Dynkin diagrams using automorphisms of the graph itself, provided that the automorphism under use does not permute adjacent vertices of the graph. This trick allows some computations for non simply-laced algebras to be carried out by viewing them as subalgebras of simply laced ones. Beyond the simple cases, it is not possible in general to completely render a graph simply laced by using Kac's trick in reverse. It is, however, possible to do any particular root space calculation using a bounded number of generators via such an argument. This phenomenon is similar in essence to the construction of the universal cover in algebraic topology, where it is not always true that the universal cover is finite, any particular loop can be undone by a finite cover.
One of the nice things that this geometric information lets you do is state a trichotomy theorem, which appears in Kac's book \cite{Kac_2003}:
\begin{thm}
    \begin{enumerate}
        \item If the matrix $DA$ is positive definite, the algebra of $A$ is a finite-dimensional, semisimple lie algebra
        \item If the matrix $DA$ is positive semidefinite, the algebra is infinite dimensional, but dimension grows polynomially in the number of generators used
        \item If the matrix $DA$ is indefinite, the algebra grows exponentially
    \end{enumerate}
\end{thm}
We narrow down our enumeration of basis elements by using the following basic facts:
\begin{prop}
Every Kac-Moody algebra has a triangular decomposition $F\oplus h\oplus E$ where \begin{itemize}
            \item $E$ contains the $e_i$s and their ad-polynomials
            \item $F$ contains the $f_j$s and their ad-polynomials
            \item $h$ is as before
        \end{itemize}
\end{prop}
\begin{proof}
    Applying an $f_i$ (resp $e_i$) to $E$ (resp $F$) gives lower degree ad-polynomials  or elements of $h$
\end{proof}

\begin{prop}
    $E\cong F$ as Lie algebras
\end{prop}

Together, these facts mean that we can study important features, such as growth rate, of the Kac-Moody algebra by only studying $E$. Thus we study the growth rate of $E$, which is a Lie algebra with generators $e_i$ and only the relations $ad_{e_i}^{1-A_{ij}} e_j = 0$ for each pair of generators. Because this is a finite presentation of the Lie Algebra, it is possible to tell algorithmically if there is a linear dependence between a finite collection of $\ad$-polynomials. The analogous problem is not solvable for groups (see Novikov \cite{Novikov1958-NOVAUO}). In this case, the fact that the relations do not create dependencies between $\ad$-polynomials with different degree (different numbers of each generator) means the search space is finite, and thus that the dimension of each root space (the \textbf{root space multiplicity}). The state of the art estimates for counting root space multiplicities can be found in a paper by Carbone et al \cite{Carbone} from 2024. While it gives estimates of root space multiplicities along different directions, which are much more fine-grained data than overall growth rate, even that paper does not give closed form formulae; the difficulty of that problem was noted as far back as Borcherds' 1986 paper on vertex algebras \cite{Borcherds}. It is known, due to Hausel's proof of Kac's conjectures \cite{Hausel2010}, that this is equivalent to finding the constant term (in $q$) in the polynomial counting absolutely indecomposable representations of the quiver over $\mathbb{F}_q$, but this problem is not a priori any easier than simply counting representatives and redundancies.

\section{Facts about the Affine Algebras}\label{sec:affinefacts}
The \textit{Affine} Kac-Moody algebras are a restricted subclass of Kac-Moody algebras which are ``one step beyond" the simple Lie algebras. They are the most well studied class of infinite dimensional Kac-Moody algebras and the most useful in applications, to the extent that in the physics literature they are sometimes simply referred to as ``the Kac-Moody algebras." 
\begin{fact}
Given a Kac-Moody algebra $\mathfrak{g}$, the following are equivalent:
\begin{itemize}
    \item $\mathfrak{g}$ is an affine algebra
    \item The Cartan matrix $A$ of $\mathfrak{g}$ is symmetrizable and $DA$ is positive semidefinite.
    \item $\mathfrak{g}$ is infinite dimensional and has polynomial growth
    \item Each irreducible direct summand of $\mathfrak{g}$ can be constructed as a central extension of a twisted loop algebra over a simple Lie algebra (a process appropriately called ``affinization")
    \item $\mathfrak{g}$ arises as the current algebra in a Wess-Zumino-Witten model
\end{itemize}
\end{fact}
The first four points can be found in Kac's book \cite{Kac_2003}, and for the fifth one (which is basically the fourth one, but highlights the relevance of these algebras in physics) should probably consult Di Francesco et al. \cite{CFT}. Given a simple Lie algebra $\mathfrak{a}$, the untwisted loop algebra is $\mathfrak{a}\otimes \C(t)$. Twisting here means that rather than sections of $\mathbb{S}^1\times\mathfrak{a}\times \mathbb{S}^1$, the sections are of the bundle given by $([0,2\pi]\times \mathfrak{a}/\cong_\phi )\rightarrow \mathbb{S}^1$, where $\cong_\phi$ is the equivalence relation identifying
\[(0,a) \cong_\phi (1,\phi(a))\]
for $\phi$ an automorphism of $\mathfrak{a}$. This immediately gives both infinite dimensionality and polynomial growth. Showing that this implies positive semidefiniteness is an exercise in finding a system of simple roots for this algebra. Showing that all positive semi-definite Cartan matrices arise in this way is a proof by exhaustion, and showing that that condition is equivalent to polynomial growth is the point of the argument that stretches from \S\ref{sec:Simple}, through \S\ref{sec:Affine}, to  \S\ref{sec:WildBinary}.

Also manifest in the affinization construction is that the associated algebra has positive roots with positive multiplicity. These are a new development that happens beyond the simple case, and are called \textit{imaginary roots}. In the affine case, the positive imaginary root spaces all have the same dimension (guaranteed by the affinization construction), which is also the dimension of the original Cartan subalgebra. This is why we see a repeating pattern in the diagrams for affine roots: the number of iterations of that pattern is the number of factors of $t$ in the loop algebra expression for that positive imaginary root. 

\nocite{baezdynkin}

\end{document}